\documentclass[11pt,reqno]{amsproc}
\usepackage{amsmath,latexsym,amssymb,verbatim,amsbsy,amsthm,mathrsfs,tikz,times}
\usepackage[margin=1in]{geometry}
\usepackage{enumitem}
\title{Global regularity for 2D Muskat equations with finite slope}

\author{Peter Constantin}
\address{Princeton University, Princeton, NJ 08544}
\email{const@math.princeton.edu}

\author{Francisco Gancedo}
\address{University of Seville, Sevilla, Spain 41012}
\email{fgancedo@us.es}

\author{Roman Shvydkoy}
\address{University of Illinois at Chicago, Chicago, IL, 60607}
\email{shvydkoy@uic.edu}

\author{Vlad Vicol}
\address{Princeton University, Princeton, NJ 08544}
\email{vvicol@math.princeton.edu}

\usepackage{color}
\usepackage[colorlinks=true, pdfstartview=FitV, linkcolor=blue, citecolor=blue, urlcolor=blue]{hyperref}

\theoremstyle{plain}
\newtheorem{theorem}{Theorem}[section]

\newtheorem{lemma}[theorem]{Lemma}

\theoremstyle{definition}

\numberwithin{equation}{section}

\def\RR{{\mathbb R}}

\def\DD{\mathcal D}
\def\RSZ{{\mathcal R}}

\def\LL{{\mathcal L}}

\newcommand{\eps}{\varepsilon}
\renewcommand{\phi}{\varphi}

\newcommand{\MOC}{modulus of continuity\;}

\newcommand{\al}{\alpha}

\begin{document}

\begin{abstract}
We consider the 2D Muskat equation for the interface between two constant density fluids in an incompressible porous medium, with velocity given by Darcy's law. We establish that as long as the slope of the interface between the two fluids remains bounded and uniformly continuous, the solution remains regular. The proofs exploit the nonlocal nonlinear parabolic nature of the equations through a series of nonlinear lower bounds for nonlocal operators. These are  used to deduce that as long as the slope of the interface remains uniformly bounded, the curvature remains bounded.  The nonlinear bounds  then allow us to obtain local existence for arbitrarily large initial data in the class $W^{2,p}$, $1<p\leq \infty$. We provide furthermore a global regularity result for small initial data: if the initial slope of the interface is sufficiently small, there exists a unique solution for all time.
\hfill \today
\end{abstract}


\subjclass[2010]{76S05,35Q35}
\keywords{porous medium, Darcy's law, Muskat problem, maximum principle}

\maketitle

\setcounter{tocdepth}{1}
\tableofcontents

\section{Introduction}\label{sec:intro}
In this paper we study solutions of the one dimensional nonlinear nonlocal equation
\begin{align}\label{eq:intro}
\partial_t f(x,t) = \int_{\RR} \frac{ \delta_\alpha f'(x,t)\alpha}{ (\delta_\alpha f(x,t))^2 + \alpha^2}  d\alpha,
\end{align}
where we have denoted by
\begin{align*}
\delta_\alpha f(x,t) &= f(x,t) - f(x-\alpha,t)
\end{align*}
the $\alpha$-finite difference of $f$ at fixed time $t \geq 0$, and by
\begin{align*}
f'(x,t)&=\partial_x f(x,t)
\end{align*}
the partial derivative of $f$ with respect to the position $x\in\RR$.

\subsection{Physical origins of the problem} Equation~\eqref{eq:intro} is derived from a two dimensional problem, where $y = f(x,t)$ describes a graph (in coordinates whereby gravity acts in the $y$ direction) which represents the interface separating two immiscible fluids of different constant densities. The flow is incompressible, and the fluids permeate a porous medium.  The 2D Darcy's law  \cite{Darcy56}
$$
\frac{\mu}{k}u(x,y,t)=-\nabla p(x,y,t)-(0,g\rho(x,y,t))
$$
is used to model the balance of gravity, internal pressure forces and velocity (replacing acceleration). Here $u$ and $p$ are the velocity and pressure of the fluid, which are functions of time $t\geq 0$ and position $(x,y)\in\RR^2$. The physical constants of the model are viscosity $\mu$, permeability $k$, and gravity $g$. For simplicity we fix these physical parameters to equal $1$. We pay special attention to the density function $\rho$, which is given by a step function representing the jump in the fluids densities. Taking the size of this jump equal to $2\pi$ (denser fluid below), we arrive at \eqref{eq:intro} (see~\cite{CordobaGancedo07} for a detailed derivation). This is the classical Muskat problem~\cite{Muskat34}, which has been broadly studied \cite{Bear72} due to its physical relevance, wide applicability, and mathematical analogy to a completely different phenomenon: the interface dynamics of fluids inside Hele-Shaw cells \cite{SaffmanTaylor58}.

The Muskat problem has been widely studied, its multiple features being taken under consideration. It can incorporate capillary forces that deal with surface tension effects~\cite{ConstantinPugh93,Chen93,KnupferMasmoudi15}, density jumps~\cite{CordobaGancedo07}, viscosity~\cite{SiegelCaflischHowinson04}, and permeability discontinuities~\cite{BerselliCordobaGranero14}. The dynamics is different for the one fluid case (with a dry region)~\cite{CastroCordobaFeffermanGancedo13}, two fluids~\cite{SiegelCaflischHowinson04}, and multi-phase flows~\cite{CordobaGancedo10}. Moreover, one may consider boundary effects~\cite{Granero14} and study the interface in the three dimensional case~\cite{Ambrose07,CordobaCordobaGancedo13}.

\subsection{Summary of known results}
The basic mathematical questions regarding \eqref{eq:intro} are the existence and regularity of solutions. The Muskat problem can be ill-posed, due to existence of unstable configurations~\cite{Otto99,Szekelyhidi12}. This feature can be understood in the contour dynamics description via the Rayleigh-Taylor condition, which involves the geometry of the density jump or the dynamics of fluids with different viscosities. Essentially, the unstable situation occurs when a denser fluid is above a less dense fluid, or when a less viscous fluid pulls a more viscous one \cite{Ambrose07}. Surface tension effects regularizes the equation~\cite{EscherSimonett97,Ambrose13}, so that the system is well-posed without satisfying the Rayleigh-Taylor condition, but there are still instabilities in this case   \cite{GuoHallstromSpirn07,EscherMatioc11}. In the stable cases it is possible to find global-in-time solutions for small initial data~\cite{SiegelCaflischHowinson04,CordobaGancedo07,ChengGraneroBelinchonShkoller14} due to the parabolicity of the equation. This smallness can be measured in different regularity classes and sizes
\cite{ConstantinCordobaGancedoStrain13,ConstantinCordobaGancedoRodriguezStrain13,Granero14,BeckSosoeWong14,ChengGraneroBelinchonShkoller14}. Low regular global existence results are obtained for Lipschitz weak solutions in the case of small slopes in \cite{ConstantinCordobaGancedoStrain13,ConstantinCordobaGancedoRodriguezStrain13}.

A very interesting phenomenon for the Muskat problem is the development of finite time singularities starting from stable initial data. One possible singularity formation happens by the breakdown of the Rayleigh-Taylor condition  in finite time \cite{CastroCordobaFeffermanGancedoLopez12}. In this case the interface cannot be parameterized by a graph of a function any longer, i.e. there exists a time $t_1>0$ when  solutions of equation \eqref{eq:intro} satisfy
$$
\lim_{t\to t_1}\|f'(t)\|_{L^\infty}=\infty.
$$
After that time the free boundary evolves in a way that produces a region with denser fluid above less dense fluid. For short time the interface is still regular (although it is not a graph) due to the fact that the parabolicity of the system gives instant analyticity. This sequence of events is termed ``wave-turning'' because it is a blow up of the graph-parameterization. Wave-turning has been proven to arise in more general cases with different geometries \cite{BerselliCordobaGranero14,GomezSerranoGraneroBelinchon14}. The regularity of the interface is lost at some time $t_2>t_1$ after wave-turning, and the interface ceases to belong to $C^4$. Thus, starting from a stable regime, the Muskat solution enters an unstable situation, and then the regularity breakdown occurs~\cite{CastroCordobaFeffermanGancedo13}. This phenomenon is far from trivial, as some unstable solutions can become stable, and again reach unstable regime \cite{CordobaGomezSerranoZlatos15}.

A different type of singularity which may develop in solutions of the Muskat problem is given by the occurrence of a finite time self-intersection of the free boundary. If the collapse is at a point, while the free boundary remains regular, then the singularity is termed a ``splash''. If the collapse is along a curve, while the free boundary remains regular, then it is termed a ``splat'' (or ``squirt singularity''). Both kinds of collisions have been ruled out in the case of a stable density jump in \cite{CordobaGancedo10,GancedoStrain13}. In the one fluid case, there exists a finite time splash blow up \cite{CastroCordobaFeffermanGancedo13Sub} in the stable regime, but the splat singularity is not possible \cite{CordobaPernasCastano14}.

\subsection{Main results}
The main purpose of this work is to develop conditional regularity results (blow up criteria) for the Muskat problem involving the boundedness of the first derivative of the interface. Our motivation is both to complement and to aide the theory of singularities for the Muskat equation by obtaining sharp regularity results in terms of purely geometric quantities such as the slope.

The bounds of nonlinear terms used to obtain the conditional regularity results also  allow us to obtain several existence results in low regularity regimes. Specifically, we prove local-in-time well-posedness and global well-posedness for Muskat solutions in classes of functions of bounded slope and $L^p$ (with $p \in (1,\infty]$) integrability of curvature.

Recently, results involving solely the second derivative of $f$ appeared in the work~\cite{ChengGraneroBelinchonShkoller14}, where the authors explore the PDE structure of the original Darcy's law in the bulk rather than the corresponding integral equation for contour dynamics \eqref{eq:intro}. They  develop an $H^2$ well-posedness theory for the equation in the case of density and viscosity jumps. Both the local and global existence results are obtained under smallness assumptions for the interface:  $H^{3/2+\epsilon}$ for the local results and $H^{2}$ for the global results. These smallness conditions imply $f' \in C^{\beta}$ for some $\beta>0$, requiring thus smallness of the H\"{o}lder continuity of the slope.

The {\em optimal} well-posedness theory (local, conditional, or global) for the Muskat equation should (conjecturally) involve only assumptions of uniform boundedness of the slope $f'$. To the best of our knowledge, this remains an outstanding open problem. The difficulty in reaching a $W^{1,\infty}$ regularity theory for $f$ may be seen in at least two ways:
\begin{itemize}[leftmargin=*]
\item When considering the evolution of $f'$, linearized around the steady flat solution, the resulting equation has $L^\infty$ as a scaling invariant norm. Therefore, for large slopes, obtaining higher regularity of the solution (required in order to obtain uniqueness) involves solving a large data {\em critical nonlinear nonlocal problem}. Of course, the additional complication is that, due to the fundamentally nonlinear nature of the equation, in the large slope regime the linearization around the flat solution rapidly ceases to be useful, and new severe difficulties arise.
\item The velocity $v$ transporting $f, f'$, and $f''$ (cf.~\eqref{eq:muskat:transport}, \eqref{eq:f'}, and \eqref{eq:f'':1} below) is obtained by computing a Calder\'on commutator applied to the identity (cf.~\eqref{eq:v:def} below). However, when $f'$ is merely bounded, while the Calder\'on commutator maps $L^p\to L^p$ for $p \in (1,\infty)$,  the endpoint $L^\infty$ inequality fails~\cite{MuscaluSchlag13b}. Moreover, since we are in 1D the $x$-derivative of $v$ enters the estimates (we cannot use incompressibility), and this term (cf.~\eqref{eq:T5:def}) yields a Calder\'on commutator applied to  $\delta_\alpha f'(x,t)/ \alpha$. Thus, to bound this requires estimates on the maximal curvature, not slope. This is the reason why up to date the existing  continuation criteria~\cite{CordobaGancedo07,CordobaGraneroBelinchonOrive14} required that the solution lies in $C^{2,\epsilon}$ for some $\epsilon>0$.
\end{itemize}

In this paper we move closer to the conjectured critical $f \in W^{1,\infty}$ well-posedness framework. The main new idea is that the maximal curvature can be controlled by either assuming uniform continuity of the slope $f'$ (cf.~part (i) of Theorem~\ref{thm:cond}), or by assuming that the initial slope $\|f_0'\|_{\infty}$ is sufficiently small (cf.~part (ii) of Theorem~\ref{thm:small}). We moreover show that under either of these conditions one can also control $\|f''\|_{L^p}$ for all $p \in (1,\infty)$. In the regime $p \in (1,2)$ not even the local existence of solutions was previously known.

The important novelty of our large data result is  that the maximal slope (and the maximal initial curvature) are allowed to be {\em arbitrarily large}, as long as the slope obeys any uniform {\em modulus of continuity}.
Our uniform $C^0$ assumption on $f'$ is a local smallness of variation assumption instead of a global smallness assumption of $f'$ in $L^\infty$.  The modulus of continuity of $f'$ is not required to vanish at the origin at any particular rate, and it can be even weaker than a Dini modulus (and thus the Calder\'on commutator in $v$ and the nonlinear terms would not necessarily yield bounded functions).

Regarding our small data result, we draw the attention to the fact that the smallness is assumed only on the initial slope, not on the $W^{2,p}$ norm.

The main tools used in our analysis are various nonlinear maximum principles for the evolution of $f''$, in the spirit of the ones previously developed in~\cite{ConstantinVicol12,ConstantinTarfuleaVicol15} for the critical SQG equation. The robustness of the pointwise and integral lower bounds available for the nonlocal operator $\LL_f$ defined in~\eqref{eq:LL:f:def} below allows us to treat all values of $p \in (1,\infty]$ and enables us to analyze the long time behavior of the curvature, cf.~\eqref{eq:curvature:decay}. The endpoint case $W^{2,1}$, which scales as $W^{1,\infty}$, remains however open.

Next we state our results more precisely. First we establish a low regularity local-existence result for \eqref{eq:intro} with initial datum that has integrable, respectively bounded, second derivatives.

\begin{theorem}[\bf Local-existence in $W^{2,p}$]\label{thm:local}
Assume the initial datum has finite energy and finite slope, that is $f_0\in L^2 \cap W^{1,\infty}$.  Let $p \in (1,\infty]$, and additionally assume that $f_0 \in W^{2,p}$. Then there exists a time $T=T(\|f_0\|_{W^{2,p}\cap W^{1,\infty}} )>0$ and a unique solution $f\in L^\infty(0,T;W^{2,p}) \cap C(0,T;L^2 \cap W^{1,\infty})$ of the initial value problem \eqref{eq:intro} with initial datum $f(x,0)=f_0(x)$.
\end{theorem}

The main result of this paper is:

\begin{theorem}[\bf Regularity criterion in terms of continuity of slope] \label{thm:cond}
Consider $f_0 \in H^k$ for $k\geq 3$ and assume that $f'$ is bounded on $[0,T]$, i.e. that
\begin{align}
\sup_{t \in [0,T]} \|f'(t)\|_{L^\infty} \leq B < \infty
\label{eq:Lip:cond}
\end{align}
for some $B>0$.
Furthermore, assume that $f'$ is uniformly continuous on $\RR \times [0,T]$. That is, assume there exits a continuous function $\rho \colon [0,\infty) \to [0,\infty)$, that is non-decreasing, bounded, with $\rho(0) = 0$, such that $f'$ obeys the modulus of continuity $\rho$, i.e. that
\begin{align}
|\delta_\alpha f'(x,t)| \leq \rho(|\alpha|)
\label{eq:f':MOC:DEF}
\end{align}
for any $x,\alpha \in \RR$ and $t \in [0,T]$.
Then the following conclusions hold:
\begin{itemize}
\item[(i)] $\sup_{t\in [0,T] } \|f''(t) \|_{L^\infty}  < C(\|f_0''\|_{L^\infty},B, \rho)$.
\item[(ii)] The solution stays regular on $[0,T]$ and $f\in C([0,T];H^k)\cap L^2(0,T;H^{k+\frac12})$.
\end{itemize}
\end{theorem}

We note that in Theorem~\ref{thm:cond} we do not require the modulus of continuity $\rho$ to vanish at $0$ at any specific rate. For example, it can be guaranteed by a condition $f' \in L^\infty(0,T;C^\beta)$ for $\beta >0$. The remarkable feature of conclusion (i), however, is that the control on $f''$ is furnished by the initial data only. If however the initial slope $f'$ is small enough, then we can show that the condition $f''\in L^p$ can be propagated without any further assumptions on the continuity of the slope. And that way we can obtain a global existence in the corresponding class.

\begin{theorem}[\bf Small slope implies global existence] \label{thm:small}
Consider initial datum $f_0 \in L^2$, with maximal slope that obeys
\begin{align}
\|f'_0\|_{L^\infty} \leq B \leq \frac{1}{C_*},
\label{eq:small:Lip:cond}
\end{align}
for a sufficiently large universal constant $C_* > 1$. Let $p\in (1,\infty]$.
If we additionally have $f_0'' \in L^{p}$, then the local in time solution obtained in Theorem~\ref{thm:local} is in fact global, and we have
\begin{align*}
\|f''(t)\|_{L^p} \leq \|f''_0\|_{L^p}
\end{align*}
for all $t>0$. Furthermore, we have that
\begin{enumerate}
\item for $p=\infty$ the curvature asymptotically vanishes as
\begin{align}
\|f''(t) \|_{L^\infty}\leq \frac{\|f''_0\|_{L^\infty}}{1 +  \frac{\|f''_0\|_{L^\infty}}{100 B} t}
\label{eq:curvature:decay}
\end{align}
for all $t \geq 0$,
\item for
$p \in [2,\infty)$ the solution obeys the $L^p$ energy inequality
\begin{align}
\|f''(t)\|_{L^p}^p + \frac{p^2}{1+B^2} \int_0^t \| |f''(s)|^{p/2} \|_{\dot{H}^{1/2}}^2 ds + \frac{1}{200B(1+B^2)} \int_0^t \|f''(s)\|_{L^{p+1}}^{p+1} ds \leq \|f''(0)\|_{L^p}^p
\label{eq:mpfpp}
\end{align}
for all $t\geq 0$,
\item for $p \in (1,2)$ the solution obeys
\begin{align}
\|f''(t)\|_{L^p}^p + \frac{1}{C B(1+B^2)} \int_0^t \|f''(s)\|_{L^{p+1}}^{p+1} ds \leq \|f''(0)\|_{L^p}^p
\label{eq:mpfpp:<2}
\end{align}
\end{enumerate}
for all $t>0$, for some sufficiently large constant $C>0$ that may depend on $p$.
\end{theorem}

We note that the assumption \eqref{eq:small:Lip:cond} together with the maximum principle for $f'$ established in~\cite[Section 5]{CordobaGancedo09} show that \eqref{eq:small:Lip:cond} holds at all times $t>0$. Now we see that bounds established in (i), (ii), (iii) also prove maximum principles at the level of the curvature. We complement existence result with the following uniqueness statement which in part shows that solutions of Theorems  \ref{thm:local}, \ref{thm:cond}, \ref{thm:small} are unique.

\begin{theorem}[\bf Uniqueness] \label{thm:uniqueness}
Let $f_0 \in L^\infty$, and assume that $f$ is a solution to \eqref{eq:intro} on time interval $[0,T]$ with the following properties: $\sup_{t\in[0,T]}\|f'(t)\|_{L^\infty} <\infty$; $f(x,t) \to 0$ as $|x| \to \infty$, for all $t\leq T$; $\partial_t f(x,t)$ exists for all $(x,t)$ and $\|\partial_t f\|_{L^\infty(\RR \times [0,T])}<\infty$; $f'$ is uniformly continuous on $\RR \times [0,T]$, i.e. \eqref{eq:f':MOC:DEF} holds for some modulus of continuity $\rho$. Then $f$ is the unique solution with $f_{0}(x)=f(x,0)$ satisfying all the listed properties.
\end{theorem}

We note that the only essential assumption in the list of Theorem \ref{thm:uniqueness} is the same uniform continuity from Theorem~\ref{thm:cond}. The remaining assumptions are present to justify the  application of the Rademacher Theorem, see the Appendix, and we believe these additional assumptions may be avoided. For instance, a sufficient condition which ensures that $\partial_t f(x,t)$ exists for all $(x,t)$ is that $f'' \in L^\infty_t W^{2,p}_x$ for some $p>1$, or that the modulus $\rho$ obeys the Dini condition.

We also note that in the course of proving Theorems \ref{thm:local} and
\ref{thm:small} we follow the standard strategy: we obtain all necessary a priori estimates and construct solutions by passing to  the limit in the regularized system as elaborated for instance in~\cite{ConstantinCordobaGancedoRodriguezStrain13}. To avoid redundancy, the proofs in this paper only consist of these a priori estimates.

\section{Preliminaries}
For the remainder of the manuscript we shall use the following notation for finite-difference quotients:
\begin{align*}
\Delta_\alpha f(x) = \frac{\delta_\alpha f(x)}{\alpha}
\end{align*}
for $x \in \RR$ and $\alpha \in \RR \setminus \{0\}$. After integration by parts \eqref{eq:intro} becomes
\begin{align}\label{eq:muskat}
\partial_t f(x,t) = P.V. \int_{\RR} \frac{f'(x,t)\alpha-\delta_\alpha f(x,t)}{ (\delta_\alpha f(x,t))^2 + \alpha^2}d\alpha
\end{align}
or equivalently,
\begin{align}\label{eq:muskat:transport}
\partial_t f(x,t) + v(x,t) \partial_x f(x,t) + P.V. \int_{\RR} \frac{\delta_\alpha f(x,t)}{ (\delta_\alpha f(x,t))^2 + \alpha^2}d\alpha = 0
\end{align}
where we have defined the  velocity field
\begin{align}
v(x,t) =- P.V. \int_{\RR} \frac{\alpha}{(\delta_\alpha f(x,t))^2 + \alpha^2} d\alpha  =- P.V. \int_{\RR} \frac{1}{(\Delta_\alpha f(x,t))^2 + 1} \frac{d\alpha}{\alpha}.
\label{eq:v:def}
\end{align}
In particular, from \eqref{eq:muskat:transport} it is immediate that the maximum principle for the global bound on $f$ holds (see \cite{CordobaGancedo09}). Moreover it is known that the maximum principle also holds for the energy (see \cite{ConstantinCordobaGancedoRodriguezStrain13}), i.e.
\begin{align}
\|f(t)\|_{L^p} \leq \|f_0\|_{L^p}
\label{eq:max:princ:Lp}
\end{align}
holds for $p=2$ and $p=\infty$.

\subsection{Equation for the first derivative}

The equation obeyed by $f'(x,t)$ is
\begin{align}\label{eq:f'}
\partial_t (f') + v(x,t) \partial_x (f') &+ P.V. \int_{\RR} \frac{\delta_\alpha f'(x,t)}{ (\delta_\alpha f(x,t))^2 + \alpha^2} d\alpha \notag\\
&= 2 \int_{\RR} \frac{ (\delta_\alpha f(x,t) - \alpha f'(x,t)) (\delta_\alpha f(x,t)) (\delta_\alpha f'(x,t))}{ \left( (\delta_\alpha f(x,t))^2 +\alpha^2 \right)^2} d\alpha.
\end{align}

\subsection{Equation for the second derivative}
The equation obeyed by the second derivative is
\begin{align}
\partial_t (f'') + v(x) \partial_x (f'') +& P.V. \int_{\RR} \frac{\delta_\alpha f''(x)}{(\delta_\alpha f(x))^2 + \alpha^2} d\alpha \notag\\
&=4\, P.V. \int_{\RR} \frac{\left(\delta_\alpha f'(x) - \alpha f''(x)\right)\delta_\alpha f(x,t)\delta_\alpha f'(x,t)}{\left( (\delta_\alpha f(x))^2 +\alpha^2 \right)^2}
  d\alpha \notag\\
 &\quad +2\, P.V. \int_{\RR}  \frac{(\delta_\alpha f(x) - \alpha f'(x)) (\delta_\alpha f'(x))^2 }{\left( (\delta_\alpha f(x))^2 +\alpha^2 \right)^2}  d\alpha \notag\\
&\quad  + 2\, P.V. \int_{\RR} \frac{(\delta_\alpha f(x) - \alpha f'(x))(\delta_\alpha f(x)) (\delta_\alpha f''(x))}{\left( (\delta_\alpha f(x))^2 +\alpha^2 \right)^2}  d\alpha \notag\\
&\quad -8\, P.V. \int_{\RR} \frac{(\delta_\alpha f(x) - \alpha f'(x))(\delta_\alpha f(x))^2 (\delta_\alpha f'(x))^2}{\left( (\delta_\alpha f(x))^2 +\alpha^2 \right)^3}
 d\alpha \notag\\
&=: T_1 + T_2 + T_3 + T_4.
\label{eq:f'':1}
\end{align}

We now pointwise multiply \eqref{eq:f'':1} by $f''(x,t)$ and obtain
\begin{align}
\left( \partial_t + v \partial_x + \LL_f \right) |f''(x,t)|^2 + \DD_f[f''](x,t) = 2 f''(x,t) \left( T_1 + T_2 + T_3 + T_4 \right)
\label{eq:f'':2}
\end{align}
where we have denoted
\begin{align}
\LL_f [g](x) = P.V. \int_{\RR} \frac{ \delta_\alpha g(x)}{ (\delta_\alpha f(x))^2 + \alpha^2} d\alpha
\label{eq:LL:f:def}
\end{align}
and
\begin{align}
\DD_f [g](x) = P.V. \int_{\RR} \frac{ (\delta_\alpha g(x))^2}{ (\delta_\alpha f(x))^2 + \alpha^2} d\alpha
\label{eq:DD:f:def}
\end{align}
for any smooth function $g$, and $T_1,\ldots,T_4$ are as given by \eqref{eq:f'':1}. It will be useful to also consider \eqref{eq:f'':2} where the transport term is written in divergence form, namely,
\begin{align}
\left( \partial_t  + \LL_f \right) |f''|^2 + \partial_x (v |f''|^2) + \DD_f[f''] = 2 f'' \left( T_1 + T_2 + T_3 + T_4 \right) + |f''|^2\, T_5
\label{eq:f'':2:div}
\end{align}
where we have denoted
\begin{align}
T_5(x,t) = \partial_x v(x,t) = 2 P.V. \int_{\RR} \frac{\delta_\alpha f'(x,t)}{\alpha^2}  \frac{\Delta_\alpha f(x,t)}{(1 + (\Delta_\alpha f(x,t))^2)^2} d\alpha.
\label{eq:T5:def}
\end{align}

Note that when $f \equiv c$, where $c$ is a constant, the above operators become
\begin{align}
\LL_c[g](x) &= P.V. \int_{\RR} \frac{ \delta_\alpha g(x)}{\alpha^2} d\alpha =: \Lambda g(x) \notag \\
\DD_c [g](x) &= P.V. \int_{\RR} \frac{ (\delta_\alpha g(x))^2}{\alpha^2} d\alpha =: \DD[g](x). \label{eq:DD}
\end{align}
We further note that
\begin{align}
\int_{\RR} \DD[g](x) dx = \|g\|_{\dot{H}^{1/2}}^2
\label{eq:H:0.5}
\end{align}
and
\begin{align}
\DD_f[g](x) \geq \frac{1}{1+ \|f'\|_{L^\infty}^2} \DD[g](x)
\label{eq:Df:DD}
\end{align}
for all $x \in \RR$.

\subsection{Evolution of the $p$th power of the second derivative, for $p \geq 2$}
When $p \in [2,\infty)$, we consider the function
\[
\phi_p (r) = |r|^{p/2},
\]
and we multiply \eqref{eq:f'':2} by $\phi_p'(|f''(x,t)|^2)$. The convexity of $\phi_p$ (which holds if $p\geq 2$) ensures that
\[
\phi_p'(|f''(x,t)|^2) \LL_f[|f''(x,t)|^2] \geq \LL_f[\phi_p(|f''(x,t)|^2)]
\]
and we thus obtain
\begin{align}
\left( \partial_t + v \partial_x + \LL_f \right) |f''(x,t)|^p + \frac{p |f''(x,t)|^{p-2}}{2} \DD_f[f''](x,t) \leq p |f''(x,t)|^{p-1} \left| T_1 + T_2 + T_3 + T_4 \right|.
\label{eq:f'':p:geq:2}
\end{align}
Moreover, we may again write this in divergence form as
\begin{align}
\left( \partial_t + \LL_f \right) |f''|^p + \partial_x( v |f''|^p) + \frac{p |f''|^{p-2}}{2} \DD_f[f''] \leq p |f''|^{p-1} \left| T_1 + T_2 + T_3 + T_4 \right| + |f''|^p \; |T_5|
\label{eq:f'':p:geq:2:div}
\end{align}
where $T_5$ is as defined in \eqref{eq:T5:def}.

\subsection{Evolution of the $p$th power of the second derivative, for $p\in(1,2)$}
When $p \in (1,2)$, the above trick does not work, since the function $\phi_p$ is in this case concave. Instead, we consider the function
\[
\psi_p(r) = |r|^p
\]
which is convex, and obeys
\begin{align}
\psi_p(s) - \psi_p(r) + \psi_p'(r) (r-s) = |s|^p - |r|^p + p \frac{|r|^p}{r} (r-s) \geq p (p-1)
\begin{cases}
\frac{(r-s)^2}{2 |r|^{2-p}}, & 2 |r| \geq |s| \\
\frac{|r-s|^p}{4}, & |s| \geq 2 |r|.
\end{cases}
\label{eq:convex}
\end{align}
for all $r,s \in \RR$. Thus, for $p \in (1,2)$ we are lead to consider the nonlocal operators
\[
\DD_{f}^p[g](x) = \int_{\RR} \frac{|g(x) - g(x-\alpha)|^p}{\alpha^2 + (\delta_\alpha f(x))^2} d\alpha
\]
and
\[
\DD^{p}[g](x) = \int_{\RR} \frac{|g(x) - g(x-\alpha)|^p}{\alpha^2} d\alpha.
\]
Note that for $p=2$, we have $\DD_2 = \DD$, and that for $g \in L^\infty \cap C^1$ the above integrals are well-defined (even without a principal value), since $p>1$. Using \eqref{eq:convex} we may show that upon multiplying \eqref{eq:f'':1} by $\psi_p'(f''(x,t)) = p |f''(x,t)|^p/f''(x,t)$, that
\begin{align}
(\partial_t + v(x,t) \partial_x + \LL_f) |f''(x,t)|^p
&+ \frac{p(p-1)}{4} \min\left\{ \frac{2 \DD_{f}[f''](x,t)}{|f''(x,t)|^{2-p}} , \DD_{f}^p[f''](x,t)\right\} \notag\\
&\quad \leq p |f''(x,t)|^{p-1} |T_1(x,t) + \ldots + T_4(x,t)|,
\label{eq:f'':p:leq:2}
\end{align}
or in divergence form,
\begin{align}
(\partial_t + \LL_f) |f''|^p  + \partial_x( v  |f''|^p )
&+ \frac{p(p-1)}{4} \min\left\{ \frac{2 \DD_f[f'']}{|f''|^{2-p}} , \DD_{f}^p[f'']  \right\} \notag\\
&\quad \leq p |f''|^{p-1} |T_1+ \ldots + T_4| + |f''|^p |T_5|,
\label{eq:f'':p:leq:2:div}
\end{align}
in analogy with \eqref{eq:f'':p:geq:2} and  \eqref{eq:f'':p:geq:2:div}.

\subsection{Bounds for Taylor expansions}
The finite difference quotient may be bounded directly as
\begin{align}
|\Delta_\alpha f(x)|\leq B
\label{eq:f':bnd}
\end{align}
for any $\alpha,x$.
We note that we may expand
\[
\RSZ_1[f'](x,\alpha ) = \Delta_\alpha f(x)  - f'(x) =  \frac{1}{\alpha} \int_{x-\alpha}^x (f'(z)-f'(x)) dz.
\]
The bound
\[
|\RSZ_1[f'](x,\alpha)| \leq 2B
\]
is immediate due to $\|f'\|_{L^\infty} \leq B$, but if additionally $f'$ has a \MOC $\rho$, then we have the improved bound
\begin{align}
|\Delta_\alpha f (x) - f'(x) | \leq \rho(|\alpha|).
\label{eq:f':MOC:easy:bound}
\end{align}
Without loss of generality we assume $\rho$ is not linear near the origin, since then the conclusion and the assumption of the theorem are identical.

It will be convenient to denote
\begin{align}
\RSZ_1[f''](x,\alpha) &:= \Delta_\alpha f'(x) - f''(x) \notag\\
&= \frac{1}{\alpha} \int_{x-\alpha}^x (f''(z) - f''(x) ) dz
\label{eq:R1:def}
\end{align}
and
\begin{align}
\RSZ_2[f''](x,\alpha) &:= \Delta_\alpha f(x) - f'(x) +\frac\alpha2 f''(x) \notag\\
&= \RSZ_1[f'](x,\alpha) + \frac \alpha 2 f''(x) \notag\\
&=\frac{1}{\alpha} \int_{x-\alpha}^{x} \int_{x}^{z} ( f''(w) - f''(x) ) dw dz.
\label{eq:R2:def}
\end{align}
as the first order and the second order expansions of $f''(x-\alpha)$ around $f''(x)$.
For these terms we have pointwise in $x$ estimates and $\alpha$ in terms of the dissipation present on the right side of \eqref{eq:f'':2}. First, we have that
\begin{align}
|\RSZ_1[f''](x,\alpha)| &\leq \frac{1}{\alpha} \int_{x-\alpha}^x \frac{|f''(z) - f''(x)|}{|z-x|} |z-x| dz \notag\\
&\leq \frac{1}{\alpha} (\DD[f''](x))^{1/2} \left( \int_{x-\alpha}^{x} |z-x|^2 dx \right)^{1/2} \notag\\
&\leq C  \alpha^{1/2} (\DD[f''](x))^{1/2}
\label{eq:R1}
\end{align}
for any $\alpha>0$, for some universal constant $C>0$. The bound for all $\alpha \neq 0$ trivially holds upon replacing $\alpha^{1/2}$ with $|\alpha|^{1/2}$.
Similarly, it follows that
\begin{align}
|\RSZ_2[f''](x,\alpha)|
&\leq \frac{1}{\alpha} \int_{x-\alpha}^{x} \int_{x}^{z} \frac{|f''(w) - f''(x)|}{|w-x|} |w-x| dw dz \notag\\
&\leq \frac{1}{\alpha} \int_{x-\alpha}^{x} \left(\int_{x}^{z} \frac{|f''(w) - f''(x)|^2}{|w-x|^2}dw \right)^{1/2}   \left(\int_{x}^{z}|w-x|^2 dw  \right)^{1/2} dz \notag\\
&\leq C  \frac{(\DD[f''](x))^{1/2}}{\alpha} \int_{x-\alpha}^{x} |z-x|^{3/2} dz \notag\\
&\leq C \alpha^{3/2} (\DD[f''](x))^{1/2}
\label{eq:R2}
\end{align}
for any $\alpha>0$, for some universal constant $C>0$. The bound for all $\alpha \neq 0$ trivially holds upon replacing $\alpha^{3/2}$ with $|\alpha|^{3/2}$. Lastly, we note that for $p\in (1,2)$ in a similar way to \eqref{eq:R1} and \eqref{eq:R2} we may bound
\begin{align}
|\RSZ_1[f''](x,\alpha)| &\leq C \alpha^{1/p} (\DD^p[f''](x))^{1/p}
\label{eq:R1:p} \\
 |\RSZ_2[f''](x,\alpha)| &\leq C \alpha^{(p+1)/p} (\DD^p[f''](x))^{1/p}
 \label{eq:R2:p}
\end{align}
with a constant $C = C(p) > 0 $.

\section{Nonlinear lower bounds}
In this section we use \eqref{eq:Lip:cond} in order to obtain lower bounds for $\DD_f[f''](x)$ at any $x \in \RR$. These lower bounds follow in a similar way to the bounds previously established in~\cite{ConstantinVicol12,ConstantinTarfuleaVicol15} for $\Lambda$. The main results of this section are.
\begin{lemma}\label{lem:lower:Lip}
 Assume that $f$ is Lipschitz continuous with Lipschitz bound $B$, i.e. that \eqref{eq:Lip:cond} holds.
 Then we have that
\begin{align}
\DD_f [f''](x) \geq \frac{1}{24 B(1+B^2)} |f''(x)|^3
\label{eq:lower:bound}.
\end{align}
holds pointwise for $x \in \RR$. Moreover, for $p \in (1,2)$, we have that
\begin{align}
\DD_f^p[f''](x) \geq \frac{1}{96 B(1+B^2)} |f''(x)|^{p+1}
\label{eq:lower:bound:Dp}
\end{align}
holds for all $x \in \RR$.
\end{lemma}
The above lower bound however will only suffice to prove a small $B$ result. For large values of $B$ we must obtain a better lower bound. If we additionally use a modulus of continuity obeyed by $f'$, we obtain:
\begin{lemma}\label{lem:lower:Lip:MOC}
 In addition to the assumption in Lemma~\ref{lem:lower:Lip}, assume that $f'$ obeys a modulus of continuity $\rho$. Then there exists a continuous  function $L_B : [0,\infty) \to [0,\infty)$ that implicitly also depends on $\rho$, such that
\begin{align}
\DD_f[f''](x) \geq L_B(|f''(x)|)
\label{eq:optimal:lower:bound}
\end{align}
holds pointwise for $x \in \RR$, and we have that
\begin{align}
\lim_{y \to \infty} \frac{L_B(y)}{y^3} = \infty
\label{eq:proof:better}
\end{align}
at a rate that depends on how fast $\lim_{r \to 0^+} \rho(r) = 0$.
\end{lemma}
Next we provide two lower bounds whose proofs follow by similar arguments to those in Lemma~\ref{lem:lower:Lip} (see also~\cite{ConstantinVicol12,ConstantinTarfuleaVicol15}).
\begin{lemma}\label{lem:lower:fpp:Lp}
Assume that $f$ is Lipschitz continuous with Lipschitz bound $B$, and let $p \in (1,\infty)$. Then the following lower bound
\begin{align}
\DD_f [f''](x) &\geq \frac{1}{8^p \|f''\|^p_{L^p}(1+B^2)} |f''(x)|^{2+p}
\label{eq:lower:fpp:Lp}
\end{align}
holds for all $x\in\RR$. Moreover, for $p \in (1,2)$ we have that
\begin{align}
\DD_f^p[f''](x) &\geq \frac{1}{128 \|f''\|^p_{L^p}(1+B^2)} |f''(x)|^{2p}
\label{eq:lower:fpp:Dp:Lp}
\end{align}
holds.
\end{lemma}
\begin{lemma}\label{lem:lower:fppp}
Assume that $f$ is a regular function with Lipschitz bound $B$. Then the following lower bound holds for any $x\in\RR$.
\begin{align}
\DD_f [f'''](x) \geq \frac{1}{24 \|f''\|_{L^\infty}(1+B^2)}|f'''(x)|^3.
\label{eq:lower:fppp}
\end{align}
\end{lemma}

\begin{proof}[Proof of Lemma~\ref{lem:lower:Lip}]
For this purpose, let $r>0$ to be chosen later, and let $\chi$ be an even cutoff function, with $\chi = 1$ on $[1,\infty)$, $\chi =0$ on $[0,1/2]$, and $\chi'=2$ on $(1/2,1)$. We then have
\begin{align}\label{eq:proof:coro}
\DD_f [f''](x)
&\geq \frac{1}{1+B^2}\DD[f''](x)\geq \frac{1}{1+B^2}\int_{\RR} \frac{|f''(x)|^2 - 2 f''(x) f''(x-\alpha)}{\alpha^2} \chi\left(\frac{\alpha}{r}\right)d\alpha \notag\\
&\geq \frac{1}{1 + B^2}\Big( |f''(x)|^2 \int_{|\alpha|\geq r} \frac{d\alpha}{\alpha^2} - 2 |f''(x)| \Big| \int_{\RR} \frac{\partial_\alpha f'(x-\alpha)}{\alpha^2}\chi\left(\frac{\alpha}{r}\right) d\alpha \Big|\Big).
\end{align}
Integration by parts yields
\begin{align*}
\DD_f [f''](x)&\geq \frac{1}{(1+B^2)}\Big(\frac{2|f''(x)|^2}{r} - 4 |f''(x)| \int_{|\alpha| \geq r/2}  \frac{|f'(x-\alpha)|}{|\alpha|^3}d\alpha\notag\\
& \qquad\qquad\qquad\qquad\qquad- 4|f''(x)| \frac{1}{r}\int_{r/2\leq |\alpha|\leq r} \frac{|f'(x-\alpha)|}{\alpha^2}d\alpha\Big) \notag\\
&\geq \frac{1}{(1+B^2)}\Big(\frac{2 |f''(x)|^2}{ r} - \frac{24B|f''(x)|}{r^2}\Big).
\end{align*}
Letting
\[
r = \frac{24 B}{|f''(x)|}
\]
in the above estimate, we arrive at the bound \eqref{eq:lower:bound}.

Let $p \in (1,2)$. In order to prove \eqref{eq:lower:bound:Dp} we appeal to the inequality
\[
|r-s|^p \geq |r|^p\left(1 - p \frac{s}{r} \right)
\]
which holds for any $r,s \in \RR$ and all $p >1$.
It follows that
\begin{align}
\DD_f^p [f''](x) &\geq \frac{1}{1+B^2} \DD^p[f''](x)
\geq \frac{1}{1+B^2} \int_{\RR} \frac{|f''(x) - f''(x-\alpha)|^p}{\alpha^2} \chi\left(\frac{\alpha}{r}\right) d\alpha \notag\\
&\geq \frac{1}{(1+B^2)} \int_{|\alpha|\geq r} \frac{|f''(x)|^p}{\alpha^2} d\alpha - \frac{p |f''(x)|^{p-1}}{(1+B^2)}\left| \int_{\RR} \frac{\partial_\alpha f'(x-\alpha)}{\alpha^2} \chi\left(\frac{\alpha}{r}\right) d\alpha\right| \notag\\
&\geq \frac{1}{1+B^2} \frac{|f''(x)|^p}{r} - \frac{24}{1+B^2} \frac{B |f''(x)|^{p-1}}{r^2}.
\label{eq:lower:bound:proof:Dp}
\end{align}
Upon choosing
\[
r = \frac{48 B }{|f''(x)|}
\]
the proof of the Lemma is completed.
\end{proof}

\begin{proof}[Proof of Lemma~\ref{lem:lower:Lip:MOC}] Similarly to the above proof, we bound $\DD_f$ from below as
\begin{align*}
\DD_f [f''](x)
&\geq \frac{1}{(1+B^2)}\Big(\frac{2|f''(x)|^2}{r} - 2 |f''(x)| \Big| \int_{\RR} \frac{\partial_\alpha \left( f'(x-\alpha) - f'(x) \right)}{\alpha^2} \chi\left(\frac{\alpha}{r}\right) d\alpha \Big|\Big) \notag\\
&\geq \frac{1}{(1+B^2)}\Big(\frac{2|f''(x)|^2}{r} - 4|f''(x)| \int_{|\alpha| \geq r/2}  \frac{|\delta_\alpha f'(x)|}{|\alpha|^3}d\alpha \notag \\
&\qquad \qquad \qquad \qquad\qquad - 4 |f''(x)| \frac{1}{r} \int_{r/2 \leq |\alpha|\leq r} \frac{|\delta_\alpha f'(x)|}{\alpha^2} d\alpha\Big) \notag\\
&\geq \frac{1}{(1+B^2)}\Big(\frac{2 |f''(x)|^2}{r} - 8|f''(x)| \int_{r/2}^\infty \frac{\rho(\alpha)}{\alpha^3} d\alpha - 8 |f''(x)| \frac{1}{r} \int_{r/2}^{r} \frac{\rho(\alpha)}{\alpha^2} d\alpha\Big) \notag\\
&\geq \frac{1}{(1+B^2)}\Big(\frac{2 |f''(x)|^2}{ r} - 12 |f''(x)| \int_{r/2}^\infty \frac{\rho(\alpha)}{\alpha^3} d\alpha\Big).
\end{align*}
In this case we choose $r = r( |f''(x)|)$ as the smallest $r>0$ which solves
\begin{align}
\frac{|f''(x)|}{6 } = r \int_{r/2}^\infty \frac{\rho(\alpha)}{\alpha^3} d\alpha. \label{eq:r:optimal:def}
\end{align}
The existence of such an $r$ is guaranteed by the intermediate value theorem, and by computing the  limits at $r \to 0+$ and $r\to\infty$ of the right side of \eqref{eq:r:optimal:def}. Indeed, these limits are
\[
\lim_{r \to 0} r \int_{r/2}^\infty \frac{\rho(\alpha)}{\alpha^3} d\alpha = + \infty
\]
since $\rho$ is not Lipschitz, and
\[
\lim_{r \to \infty} r \int_{r/2}^\infty \frac{\rho(\alpha)}{\alpha^3} d\alpha = 0.
\]
 Note that with this choice we have
\begin{align}
\lim_{y\to \infty} r(y) = 0.
\label{eq:r:asymptotics}
\end{align}
With the choice \eqref{eq:r:optimal:def}, the lower bound obtained on $\DD_f [f'']$ becomes
\begin{align}
\DD_f[f''](x) \geq L_B(|f''(x)|),
\label{eq:optimal:lower:bound:*}
\end{align}
where the function $L_B(y)$ is defined on $(0,\infty)$ implicitly by
\begin{align}
L_B(y) = \frac{y^2}{(1+B^2) r(y)}, \qquad \mbox{with} \qquad r(y) \int_{r(y)/2}^\infty  \frac{ \rho(\alpha) }{\alpha^3} d\alpha = \frac{y}{6}.
\label{eq:L:B:def}
\end{align}
A short computation shows that since $r(y)$ is  Lipschitz continuous on $(0,\infty)$, which implies that $L_B$ is also Lipschitz continuous on this interval. The important aspect to notice is that since $\rho(0) = 0$, we have
\begin{align}
\lim_{y\to \infty} \frac{L_B(y)}{y^3} = \infty,
\label{eq:proof:better:*}
\end{align}
which means that the lower bound \eqref{eq:optimal:lower:bound:*} is sharper than \eqref{eq:lower:bound}, when $|f''(x)| \gg 1$.

In order to prove \eqref{eq:proof:better:*}, note that
\[
\frac{L_B(y)}{y^3} = \frac{1}{(1+B^2) y r(y)}
\]
and
\[
\frac{y r(y)}{6} = r(y)^2 \int_{r(y)/2}^\infty  \frac{ \rho(\alpha) }{\alpha^3} d\alpha.
\]
Therefore, in view of \eqref{eq:r:asymptotics}, the limit in \eqref{eq:proof:better:*} indeed diverges if we establish that
\begin{align}
\lim_{r \to 0^+} r^2 \int_{r/2}^{\infty} \frac{\rho(\alpha)}{\alpha^3} d\alpha = 0.
\label{eq:lower:to:check}
\end{align}
Clearly, it is sufficient to verify that
\[
\lim_{r \to 0^+} r^2 \int_{r}^{\sqrt{r}} \frac{\rho(\alpha)}{\alpha^3} d\alpha = 0.
\]
Indeed, since $\rho$ is a modulus of continuity of $f'$ and we have that $\|f'\|_{L^\infty} \leq B$, we obtain
\[
r^2 \int_{\sqrt{r}}^{\infty} \frac{\rho(\alpha)}{\alpha^3} d\alpha \leq 2B r^2  \int_{\sqrt{r}}^{\infty} \frac{1}{\alpha^3} d\alpha \leq Br \to 0 \qquad \mbox{as} \qquad r\to 0.
\]
On the other hand, it is easy to see from the monotonicity of $\rho$ that
\[
r^2 \int_{r}^{\sqrt{r}} \frac{\rho(\alpha)}{\alpha^3} d\alpha  \leq \rho(\sqrt{r}) r^2 \int_{r}^{\infty} \frac{d\alpha}{\alpha^3} = \frac{\rho(\sqrt{r})}{2} \to 0\qquad \mbox{as} \qquad r\to 0
\]
since $\rho(0) = 0$.
\end{proof}

\begin{proof}[Proof of Lemma ~\ref{lem:lower:fpp:Lp}]
In order to prove \eqref{eq:lower:fpp:Lp} we proceed as in \eqref{eq:proof:coro}, but we do not integrate by parts in the second term. This yields
\begin{align*}
\DD_f[f''](x) &\geq \frac{2}{1+B^2} \frac{|f''(x)|^2}{r} -  \frac{2}{1+B^2} |f''(x)| \|f''\|_{L^p} \left( \int_{|\alpha|\geq r/2} \frac{1}{|\alpha|^{2p/(p-1)}} \right)^{(p-1)/p} \notag\\
&\geq \frac{2}{1+B^2} \frac{|f''(x)|^2}{r} -  \frac{8}{1+B^2}  \frac{|f''(x)| \|f''\|_{L^p}}{r^{(p+1)/p}}
\end{align*}
for all $p >1$. The desired inequality follows upon choosing
\[
r= \frac{8^p \|f''\|_{L^p}^p}{|f''(x)|^p}.
\]

Similarly, for $p\in(1,2)$ in order to prove \eqref{eq:lower:fpp:Dp:Lp},  we consider \eqref{eq:lower:bound:proof:Dp}, but we do not integrate by parts in the second term.
We obtain
\[
\DD_f^p[f''](x) \geq \frac{1}{1+B^2} \frac{|f''(x)|^p}{r} - \frac{4}{1+B^2} \frac{|f''(x)|^{p-1} \|f''\|_{L^p}}{r^{(p+1)/p}}.
\]
Choosing exactly as above concludes the proof of the Lemma.
\end{proof}

\begin{proof}[Proof of Lemma~\ref{lem:lower:fppp}]
We consider the inequality analogous to \eqref{eq:proof:coro} with $f''$ replaced by $f'''$. After integrating by parts once, the same argument used to prove Lemma~\ref{lem:lower:Lip} yields \eqref{eq:lower:fppp}.
\end{proof}

\section{Bounds for the nonlinear terms}
In this section we give pointwise in $x$ bounds for the nonlinear terms $T_i(x)$, with $i \in \{1,\ldots,5\}$ appearing on the right sides of \eqref{eq:f'':2} and \eqref{eq:f'':2:div}.
We first  fix a small constant
\[
\eps  \in (0,1]
\]
to be chosen later. The bounds we obtain depend on this $\eps$, and $\eps$ will be chosen differently in the proofs of Theorem~\ref{thm:local}, Theorem \ref{thm:cond}, and Theorem~\ref{thm:small} respectively.
The main result of this section is:
\begin{lemma}
\label{lem:T:1:2:3:4:5}	
Let $B > 0$ be such that \eqref{eq:Lip:cond} holds and fix $\eps \in (0,1]$. There exists a positive universal constant $C>0$ such that the bounds
\begin{align}
 |T_1(x)| + |T_2(x)| + |T_3(x)| + |T_4(x)|  &\leq \frac{CB}{\eps^2}  |f''(x)|^{2} + C \eps B^2 \frac{\DD[f''](x)}{|f''(x)|}
\label{eq:T:1:2:3:4} \\
|T_5(x)| &\leq C B |\Lambda f'(x)| + \frac{C B}{\eps}  |f''(x)| + C \eps^2 B^2 \frac{\DD[f''](x)}{|f''(x)|^2}
\label{eq:T:5}
\end{align}
hold for all $x \in \RR$ where $f''(x) \neq 0$.
\end{lemma}

In addition, we need a pointwise estimate for the nonlinear terms in terms of the dissipative term $\DD^p[f'']$, when $p \in (1,2)$.

\begin{lemma}
\label{lem:T:1:2:3:4:5:p}
Let $B > 0$ be such that \eqref{eq:Lip:cond} holds, let $p \in (1,2)$, and fix $\eps \in (0,1]$. There exists a positive universal constant $C>0$ such that the bounds
\begin{align}
 |T_1(x)| + |T_2(x)| &+ |T_3(x)| + |T_4(x)|  \leq \frac{C B}{\eps^{2/(p-1)}}  |f''(x)|^2 + C \eps B^2 \frac{\DD^p[f''](x)}{|f''(x)|^{p-1}}
\label{eq:T:1:2:3:4:p} \\
|T_5(x)| &\leq C B |\Lambda f'(x)| + \frac{C B}{\eps^p}   |f''(x)| + C \eps B^2 \frac{\DD^p[f''](x)}{|f''(x)|^{p}}
\label{eq:T:5:p}
\end{align}
hold for all $x \in \RR$ where $f''(x) \neq 0$.
\end{lemma}

\begin{proof}[Proof of Lemma~\ref{lem:T:1:2:3:4:5}]
Throughout  this proof, we fix a cutoff radius (for $\alpha$)
\begin{align}
\eta = \eta(x) = \frac{\eps B}{|f''(x)|}.
\label{eq:eta:def}
\end{align}
Note that at the points $x$ where $f''(x) = 0$, there is no estimate to be done for $T_i(x)$ terms as they are multiplied with $f''(x)$ in \eqref{eq:f'':2} and  \eqref{eq:f'':2:div}.

{\bf Estimate for the $T_1$ term.\ }
We decompose $T_1$ into an inner piece and an outer piece according to
\begin{align*}
T_1(x)
&= 4\; P.V. \int_{\RR} \frac{\left(\delta_\alpha f'(x) - \alpha f''(x)\right)\; \delta_\alpha f(x)\; \delta_\alpha f'(x)}{\left( (\delta_\alpha f(x))^2 +\alpha^2 \right)^2} d\alpha \notag\\
&= 4\; P.V. \int_{|\alpha| \leq \eta} \frac{\RSZ_1[f''](x,\alpha) \; \Delta_\alpha f(x) \; \left(f''(x) + \RSZ_1[f''](x,\alpha) \right)}{\left( (\Delta_\alpha f(x))^2 +1 \right)^2 \alpha} d\alpha \notag\\
&\qquad + 4  \int_{|\alpha|> \eta} \frac{\RSZ_1[f''](x,\alpha) \; \Delta_\alpha f(x) \; \delta_\alpha f'(x)}{\left( (\Delta_\alpha f(x))^2 + 1 \right)^2 \alpha^2} d\alpha \notag\\
&=: T_{1,in}(x) + T_{1,out}(x).
\end{align*}
Using the pointwise in $x$ and $\alpha$ bounds \eqref{eq:f':bnd} and \eqref{eq:R1}, and the definition of $\eta$ in \eqref{eq:eta:def}, we obtain
\begin{align}
|T_{1,in}(x)|
&\leq C B |f''(x)| ( \DD[f''](x) )^{1/2} \int_{|\alpha| \leq \eta} \frac{d\alpha}{|\alpha|^{1/2}}
+ C B \DD[f''](x)  \int_{|\alpha| \leq \eta} d\alpha \notag\\
&\leq C \eps^{1/2} B^{3/2}  |f''(x)|^{1/2} ( \DD[f''](x) )^{1/2}
+ C \eps B^2  \frac{\DD[f''](x)}{|f''(x)|} \notag\\
&\leq C B  |f''(x)|^{2}
+ C \eps B^2  \frac{\DD[f''](x)}{|f''(x)|}
\label{eq:T:1:in}
\end{align}
and
\begin{align}
|T_{1,out}(x)| &\leq C B^2 (\DD[f''](x))^{1/2} \int_{|\alpha|>\eta} \frac{d\alpha}{|\alpha|^{3/2}} \notag\\
&\leq C \eps^{-1/2} B^{3/2} |f''(x)|^{1/2} (\DD[f''](x))^{1/2} \notag\\
&\leq C \eps^{-2} B  |f''(x)|^{2}  + C \eps B^2  \frac{\DD[f''](x)}{|f''(x)|}
\label{eq:T:1:out}
\end{align}
for some universal constant $C>0$. Combining \eqref{eq:T:1:in} and \eqref{eq:T:1:out} we arrive at
\begin{align}
2 \left|f''(x) T_1(x) \right| \leq C \eps^{-2} B |f''(x)|^{3} + C \eps B^2 \DD[f''](x)
\label{eq:T:1:final}
\end{align}
for some universal $C>0$.  This bound is consistent with \eqref{eq:T:1:2:3:4}.

{\bf Estimate for the $T_2$ term.\ }
We decompose $T_2$ as
\begin{align*}
T_2 &= 2\; P.V. \int_{\RR}  \frac{(\Delta_\alpha f(x) - f'(x)) (\delta_\alpha f'(x))^2 }{\left( (\Delta_\alpha f(x))^2 + 1 \right)^2 \alpha^3}  d\alpha \notag\\
&=2\; P.V. \int_{\RR}  \frac{(-\alpha f''(x)/2 + \RSZ_2[f''](x,\alpha)) (\delta_\alpha f'(x))^2 }{\left( (\Delta_\alpha f(x))^2 + 1 \right)^2 \alpha^3}  d\alpha \notag\\
&= 2\; P.V. \int_{|\alpha|\leq\eta}  \frac{  \RSZ_2[f''](x,\alpha) (\delta_\alpha f'(x)) (f''(x) + \RSZ_1[f''](x,\alpha)) }{\left( (\Delta_\alpha f(x))^2 + 1 \right)^2 \alpha^2}  d\alpha \notag \\
&\qquad - f''(x) \; P.V. \int_{|\alpha|\leq\eta}  \frac{ (f''(x) + \RSZ_1[f''](x,\alpha))^2 }{\left( (\Delta_\alpha f(x))^2 + 1 \right)^2}  d\alpha \notag \\
&\qquad + 2  \int_{|\alpha|>\eta}  \frac{(-\alpha f''(x) /2 + \RSZ_2[f''](x,\alpha)) (\delta_\alpha f'(x))^2 }{\left( (\Delta_\alpha f(x))^2 + 1 \right)^2 \alpha^3}  d\alpha \notag\\
&=:T_{2,1,in} + T_{2,2,in} + T_{2,out}.
\end{align*}
By appealing to \eqref{eq:R2} and \eqref{eq:R1} we may estimate the inner terms as
\begin{align}
|T_{2,1,in}| &\leq C B  (\DD[f''](x))^{1/2} \left( |f''(x)| \int_{|\alpha| \leq \eta} \frac{d\alpha}{|\alpha|^{1/2}} + (\DD[f''](x))^{1/2} \int_{|\alpha| \leq \eta} d\alpha \right) \notag\\
&\leq C \eps^{1/2} B^{3/2} |f''(x)|^{1/2} (\DD[f''](x))^{1/2} + C  \eps B^2 \frac{\DD[f''](x) }{|f''(x)|} \notag\\
&\leq C B |f''(x)|^2 + C  \eps B^2 \frac{\DD[f''](x) }{|f''(x)|}
\label{eq:T:2:1:in}
\end{align}
and
\begin{align}
|T_{2,2,in}| &\leq C|f''(x)| \left( |f''(x)|^2 \int_{|\alpha|\leq \eta}\!\! d\alpha + |f''(x)| (\DD[f''](x))^{1/2} \int_{|\alpha|\leq \eta}\!\! |\alpha|^{1/2} d\alpha + \DD[f''](x) \int_{|\alpha|\leq \eta}\!\! |\alpha| d\alpha\right) \notag\\
&\leq C \eps B |f''(x)|^2 + C \eps^{3/2} B^{3/2} |f''(x)|^{1/2} (\DD[f''](x))^{1/2} + C \eps^2 B^{2}  \frac{\DD[f''](x)}{|f''(x)|}\notag\\
&\leq C B |f''(x)|^2 + C \eps B^{2} \frac{\DD[f''](x) }{|f''(x)|}
\label{eq:T:2:2:in}
\end{align}
while the outer terms may be bounded as
\begin{align}
|T_{2,out}|
&\leq C B^2 \left( |f''(x)| \int_{|\alpha|>\eta} \frac{d\alpha}{|\alpha|^2}+ (\DD[f''](x))^{1/2} \int_{|\alpha|>\eta} \frac{d\alpha}{|\alpha|^{3/2}} \right)\notag\\
&\leq C \eps^{-1} B  |f''(x)|^2 + C \eps^{-1/2} B^{3/2} |f''(x)|^{1/2} (\DD[f''](x))^{1/2} \notag\\
&\leq C \eps^{-2} B  |f''(x)|^2 + C  \eps B^2 \frac{\DD[f''](x) }{|f''(x)|}
\label{eq:T:2:out}
\end{align}
for some universal $C>0$. Combining \eqref{eq:T:2:1:in}, \eqref{eq:T:2:2:in}, \eqref{eq:T:2:out}, using the identity \eqref{eq:H:0.5} and the Cauchy-Schwartz inequality we arrive at
\begin{align}
2 \left|  f''(x) T_2(x)  \right| \leq C \eps^{-2} B |f''(x)|^{3} + C \eps B^2 \DD[f''](x)
\label{eq:T:2:final}
\end{align}
for some universal $C>0$. This bound is consistent with \eqref{eq:T:1:2:3:4}.

{\bf Estimate for the $T_3$ term.\ }
We bound $T_3$ as
\begin{align*}
|T_3| &= 2 \left| P.V. \int_{\RR} \frac{(\Delta_\alpha f(x) -  f'(x))}{\alpha}
\frac{\Delta_\alpha f(x)}{\left( (\Delta_\alpha f(x))^2 + 1 \right)^{2}}  \frac{\delta_\alpha f''(x)}{\alpha}d\alpha \right| \notag\\
&\leq  CB (\DD[f''](x))^{1/2} \left( \int_{\RR} \frac{(\Delta_\alpha f(x) -  f'(x))^2}{\alpha^2} \right)^{1/2} \notag\\
&\leq  CB(\DD[f''](x))^{1/2} \left( \int_{|\alpha|\leq \eta} \frac{(\alpha f''(x)/2 + \RSZ_2[f''](x,\alpha])^2}{\alpha^2} \right)^{1/2} \notag\\
&\qquad + CB (\DD[f''](x))^{1/2} \left( \int_{|\alpha|>\eta} \frac{(\RSZ_1[f'](x,\alpha))^2}{\alpha^2} \right)^{1/2}\notag\\
&=: T_{3,in} + T_{3,out}.
\end{align*}
Using \eqref{eq:R2} we may further estimate
\begin{align}
T_{3,in}
&\leq CB (\DD[f''](x))^{1/2}\left( \int_{|\alpha|\leq \eta} \frac{(\alpha f''(x)/2 + \RSZ_2[f''](x,\alpha]))^2}{\alpha^2} \right)^{1/2} \notag\\
&\leq CB (\DD[f''](x))^{1/2} \notag\\
&\qquad \times \left(|f''(x)|^2 \int_{|\alpha|\leq \eta} d\alpha + |f''(x)| (\DD[f''](x))^{1/2} \int_{|\alpha|\leq \eta} |\alpha|^{1/2} d\alpha + \DD[f''](x) \int_{|\alpha|\leq\eta} |\alpha| d\alpha   \right)^{1/2} \notag\\
&\leq CB (\DD[f''](x))^{1/2} \left(B \eps |f''(x)| + B^{3/2} \eps^{3/2} \frac{(\DD[f''](x))^{1/2}}{|f''(x)|^{1/2}} + B^2 \eps^2 \frac{\DD[f''](x)}{|f''(x)|^2} \right)^{1/2} \notag\\
&\leq CB \eps^{1/2} B^{1/2} |f''(x)|^{1/2} (\DD[f''](x))^{1/2} + C \eps B^2 \frac{ \DD[f''](x) }{|f''(x)|}\notag\\
&\leq CB |f''(x)|^2 + C \eps B^2 \frac{ \DD[f''](x) }{|f''(x)|}.
\label{eq:T:3:in}
\end{align}
Similarly, we have that
\begin{align}
T_{3,out} &\leq C B (\DD[f''](x))^{1/2} \left( \int_{|\alpha|>\eta} \frac{d\alpha}{\alpha^2} \right)^{1/2} \notag\\
&\leq C \eps^{-1/2} B^{3/2} |f''(x)|^{1/2} (\DD[f''](x))^{1/2}\notag\\
&\leq  C \eps^{-2}B|f''(x)|^2 + C \eps B^2 \frac{ \DD[f''](x) }{|f''(x)|}.
\label{eq:T:3:out}
\end{align}
Combining \eqref{eq:T:3:in}--\eqref{eq:T:3:out}, leads to the estimate
\begin{align}
2 \left|  f''(x) T_3(x) \right| \leq C \eps^{-2}B |f''(x)|^3 + C \eps B^2 \DD[f''](x)
\label{eq:T:3:final}
\end{align}
for some positive universal constant $C>0$. This bound is consistent with \eqref{eq:T:1:2:3:4}.

{\bf Estimate for the $T_4$ term.\ }
We decompose $T_4$ as
\begin{align*}
T_4 &=  -8\; P.V. \int_{\RR} \frac{(\Delta_\alpha f(x) - f'(x))(\Delta_\alpha f(x))^2 (\delta_\alpha f'(x))^2}{\left( (\Delta_\alpha f(x))^2 +1 \right)^3 \alpha^3}
 d\alpha \notag\\
&= 4 f''(x) \; P.V. \int_{|\alpha|\leq \eta} \frac{ (\Delta_\alpha f(x))^2 (f''(x) + \RSZ_1[f''](x,\alpha))^2}{\left( (\Delta_\alpha f(x))^2 +1 \right)^3 }
 d\alpha \notag\\
&\qquad  -8\; P.V. \int_{|\alpha|\leq \eta} \frac{( \RSZ_2[f''](x,\alpha))(\Delta_\alpha f(x))^2 (\delta_\alpha f'(x)) (f''(x) + \RSZ_1[f''](x,\alpha))}{\left( (\Delta_\alpha f(x))^2 +1 \right)^3 \alpha^2 }
 d\alpha \notag\\
 &\qquad -8 \int_{|\alpha|>\eta} \frac{(\RSZ_1[f'](x,\alpha))(\Delta_\alpha f(x))^2 (\delta_\alpha f'(x))^2}{\left( (\Delta_\alpha f(x))^2 +1 \right)^3 \alpha^3}
 d\alpha \notag\\
 &=: T_{4,1,in} + T_{4,2,in} + T_{4,out}.
\end{align*}
Using \eqref{eq:R1} we may estimate
\begin{align}
|T_{4,1,in}|
&\leq C  |f''(x)| \notag\\
&\qquad \times \left(|f''(x)|^2 \int_{|\alpha|\leq \eta} d\alpha + |f''(x)| (\DD[f''](x))^{1/2} \int_{|\alpha|\leq \eta} |\alpha|^{1/2} d\alpha + \DD[f''](x) \int_{|\alpha|\leq \eta} |\alpha| d\alpha \right) \notag\\
&\leq C  |f''(x)| \left(\eps B |f''(x)| + \eps^{3/2} B^{3/2} \frac{(\DD[f''](x))^{1/2}}{|f''(x)|^{1/2}} + \eps^2 B^2 \frac{\DD[f''](x)}{|f''(x)|^2} \right)\notag\\
&\leq C \eps B |f''(x)|^2 + C \eps^2 B^2 \frac{\DD[f''](x)}{|f''(x)|}.
\label{eq:T:4:1:in}
\end{align}
By also appealing to \eqref{eq:R2} we obtain the bound
\begin{align}
|T_{4,2,in}|
&\leq C B (\DD[f''](x))^{1/2} \left( |f''(x)| \int_{|\alpha|\leq \eta} \frac{d\alpha}{|\alpha|^{1/2}} + (\DD[f''](x))^{1/2} \int_{|\alpha|\leq \eta} d\alpha \right) \notag\\
&\leq C \eps^{1/2} B^{3/2} |f''(x)|^{1/2} (\DD[f''](x))^{1/2} + C \eps B^2 \frac{\DD[f''](x)}{|f''(x)|}.
\label{eq:T:4:2:in}
\end{align}
On the other hand, for the outer term we obtain
\begin{align}
|T_{4,out}|
&\leq C B^3 \int_{|\alpha|>\eta} \frac{d\alpha}{|\alpha|^3} \leq C \eps^{-2} B |f''(x)|^2.
\label{eq:T:4:out}
\end{align}
Combining \eqref{eq:T:4:1:in}--\eqref{eq:T:4:out}, and using the Cauchy-Schwartz inequality we arrive at
\begin{align}
2 \left|  f''(x) T_4(x)  \right| \leq C \eps^{-2} B |f''(x)|^3 + C \eps B^2  \DD[f''](x)
\label{eq:T:4:final}
\end{align}
for some positive universal constant $C>0$. This bound is consistent with \eqref{eq:T:1:2:3:4}.

{\bf Estimate for the $T_5$ term.\ }
Recall that $T_5$ is computed by taking the derivative of \eqref{eq:v:def} as
\begin{align*}
T_5(x) &= \partial_x v(x) = - \frac{\partial}{\partial x} \left( P.V. \int_{\RR} \frac{1}{(\Delta_\alpha f(x))^2 + 1} \frac{d\alpha}{\alpha}\right) \notag\\
&= 2 P.V. \int_{\RR} \frac{\Delta_\alpha f(x)}{((\Delta_\alpha f(x))^2 + 1)^2} \frac{\delta_\alpha f'(x)}{\alpha^2} d\alpha.
\end{align*}
The main idea here is that in order to decompose $T_5$ into an inner and an outer term, we first need to subtract and add
\[
T_{5,pv}(x) = 2  \frac{f'(x)}{(f'(x))^2 +1)^2} P.V.\int_{\RR}\frac{\delta_\alpha f'(x)}{\alpha^2} d\alpha = 2 \frac{f'(x)}{(f'(x))^2 +1)^2} \Lambda f'(x).
\]
The contribution from this term is bounded by
\begin{align}
|T_{5,pv}(x)|  \leq C B |\Lambda f'(x)| = C B |H f''(x)|
\label{eq:T:5:pv}
\end{align}
where $H$ is the Hilbert transform. The difference is now decomposed further as
\[
T_5 - T_{5,pv} =  P.V. \int_{\RR} K_f(x,\al) \;(\Delta_\alpha f(x)-f'(x)) \frac{(\Delta_\alpha f'(x))}{\alpha} d\alpha,
\]
where
\begin{align*}
K_f(x,\al)=-2\frac{(\Delta_\alpha f(x))^3f'(x)+(\Delta_\alpha f(x))^2(f'(x))^2+f'(x)(\Delta_\alpha f(x))^3+2f'(x)\Delta_\alpha f(x)-1}{((\Delta_\alpha f(x))^2+1)^2(f'(x))^2+1)^2}.
\end{align*}
Next we use that
$$
|K_f(x,\al)|\leq 2, \qquad \mbox{for any value of } \Delta_\alpha f(x) \mbox{ and } f'(x).
$$
Decomposing further, it is possible to find
\begin{align*}
T_5 - T_{5,pv}&=P.V. \int_{|\alpha|<\eta} K_f(x,\al)\frac{(-\alpha f''(x)/2 + \RSZ_2[f''](x,\alpha))}{((\Delta_\alpha f(x))^2 + 1)^2} \; \frac{(f''(x) + \RSZ_1[f''](x,\alpha))}{\alpha} d\alpha \notag\\
&\qquad   +\int_{|\alpha|>\eta} K_f(x,\al)\frac{\RSZ_1[f'](x,\alpha)}{((\Delta_\alpha f(x))^2 + 1)^2} \; \frac{(\delta_\alpha f'(x))}{\alpha^2} d\alpha \notag\\
&=:T_{5,in} + T_{5,out}.
\end{align*}
For the outer term we directly obtain
\begin{align}
|T_{5,out}| &\leq C B^2 \int_{|\alpha|>\eta} \frac{d\alpha}{|\alpha|^2}  \leq \frac{C B}{\eps} |f''(x)|
\label{eq:T:5:out}.
\end{align}
We recall that $\eta=\eta(x)=\eps B |f''(x)|^{-1}$.
For the inner term, we appeal to \eqref{eq:R2} and \eqref{eq:R1} and obtain that
\begin{align}
|T_{5,in}|
&\leq |f''(x)|^2 \int_{|\alpha|<\eta} d\alpha + C |f''(x)| (\DD[f''](x))^{1/2} \int_{|\alpha|<\eta} |\alpha|^{1/2} d\alpha + C \DD[f''](x) \int_{|\alpha|<\eta} |\alpha| d\alpha \notag\\
&\leq C \eps B |f''(x)| + C \eps^{3/2}B^{3/2}(\DD[f''](x))^{1/2}  \frac{1}{|f''(x)|^{1/2}} + C\eps^2 B^2 \DD[f''](x) \frac{1}{|f''(x)|^2}
\label{eq:T:5:in}.
\end{align}
Summarizing, \eqref{eq:T:5:out} and \eqref{eq:T:5:in} and using the Cauchy-Schwartz inequality, we obtain the desired bound
\begin{align*}
|T_{5}(x) - T_{5,pv}(x)|  \leq C \frac{B}{\eps} |f''(x)| + C\eps^2 B^2 \frac{\DD[f''](x)}{|f''(x)|^2}
\end{align*}
which combined with \eqref{eq:T:5:pv} yields the desired bound \eqref{eq:T:5}.
\end{proof}

\begin{proof}[Proof of Lemma~\ref{lem:T:1:2:3:4:5:p}]

The proof closely follows that of Lemma~\ref{lem:T:1:2:3:4:5}, but uses \eqref{eq:R1:p}--\eqref{eq:R2:p} instead of \eqref{eq:R1}--\eqref{eq:R2}.
We let $\bar \eps \in (0,1]$ to be determined, and as in \eqref{eq:eta:def} let
\[
\eta = \eta(x) = \frac{\bar \eps B}{|f''(x)|}.
\]
In this proof the constant $C$ may change from line to line, and may depend on $p$, but not on $\bar \eps, B$, or $x$.

{\bf Estimate for the $T_1, T_2, T_3$, and $T_4$ terms.\ }
We first claim that
\begin{align}
|T_1(x)| + |T_2(x)| + |T_3(x)| + |T_4(x)| \leq \frac{C B}{\bar \eps^2} |f''(x)|^2 + C B^{(p+2)/p} \bar \eps^{2/p}  \frac{(\DD^p[f''](x))^{2/p}}{|f''(x)|^{2/p}}
\label{eq:T:1:2:3:4:p:p}
\end{align}
for some constant $C>0$, and all $x \in \RR$.

We verify this estimate by checking each term individually.
For $T_1$, similarly to \eqref{eq:T:1:in} and \eqref{eq:T:1:out}, by using \eqref{eq:R1:p}--\eqref{eq:R2:p} we have that
\begin{align*}
|T_{1,in}(x)| &\leq C \bar \eps^{1/p} B^{(p+1)/p} |f''(x)|^{(p-1)/p} (\DD^p[f''](x))^{1/p} + C B^{(p+2)/p} \bar \eps^{2/p}  \frac{(\DD^p[f''](x))^{2/p}}{|f''(x)|^{2/p}} \\
|T_{1,out}(x)| &\leq \frac{C B^{(p+1)/p}}{\bar \eps^{(p-1)/p}} |f''(x)|^{(p-1)/p} (\DD^p[f''](x))^{1/p}
\end{align*}
so that by Cauchy-Schwartz we obtain that $T_1$ is bounded as in \eqref{eq:T:1:2:3:4:p:p}.

For the $T_2$ term, as in \eqref{eq:T:2:1:in}, \eqref{eq:T:2:2:in}, and \eqref{eq:T:2:out}, but using \eqref{eq:R1:p}--\eqref{eq:R2:p}, we obtain
\begin{align*}
|T_{2,1,in}(x)| &\leq C \bar \eps^{1/p} B^{(p+1)/p} |f''(x)|^{(p-1)/p} (\DD^p[f''](x))^{1/p} + C B^{(p+2)/p} \bar \eps^{2/p}  \frac{(\DD^p[f''](x))^{2/p}}{|f''(x)|^{2/p}}\\
|T_{2,2,in}(x)| &\leq C  \bar \eps B  |f''(x)|^2 + C B^{(p+2)/p} \bar \eps^{(p+2)/p}  \frac{(\DD^p[f''](x))^{2/p}}{|f''(x)|^{2/p}} \\
|T_{2,out}(x)| &\leq \frac{C B}{\bar \eps} |f''(x)|^2 +  \frac{C  B^{(p+1)/p}}{\bar \eps^{(p-1)/p}} |f''(x)|^{(p-1)/p} (\DD^p[f''](x))^{1/p}
\end{align*}
so that by Cauchy-Schwartz it follows that $T_2$ obeys the bound \eqref{eq:T:1:2:3:4:p:p}.

For $T_3$, we proceed slightly differently from \eqref{eq:T:1:in} and \eqref{eq:T:1:out}, but still use \eqref{eq:R1:p}--\eqref{eq:R2:p} and obtain
\begin{align*}
|T_3(x)| &= 2 \left| P.V. \int_{\RR} \frac{\RSZ_1[f'](x,\alpha)}{|\alpha|^{2-2/p}} \frac{\Delta_\alpha f(x)}{((\Delta_\alpha f(x))^2 + 1)^2} \frac{\delta_\alpha f''(x)}{|\alpha|^{2/p}} d\alpha \right| \\
&\leq C B (\DD^p[f''](x))^{1/p} \left( \int_{\RR}  \frac{(\RSZ_1[f'](x,\alpha))^{p/(p-1)}}{|\alpha|^{2}} \right)^{(p-1)/p}\\
&\leq C B (\DD^p[f''](x))^{1/p} \left( \left( \int_{|\alpha|\leq \eta}  \frac{(\RSZ_2[f''](x,\alpha) + \alpha f''(x)/2)^{p/(p-1)}}{|\alpha|^{2}} \right)^{(p-1)/p} + B \eta^{-(p-1)/p} \right) \\
&\leq C B (\DD^p[f''](x))^{1/p} \left(|f''(x)| \eta^{1/p} + (\DD^p[f''](x))^{1/p}  \eta^{2/p} + B \eta^{-(p-1)/p} \right) \\
&\leq  \frac{C B^{(p+1)/p}}{\bar \eps^{(p-1)/p}}  |f''(x)|^{(p-1)/p} (\DD^p[f''](x))^{1/p} + C \bar \eps^{2/p} B^{(p+2)/p}  \frac{(\DD^p[f''](x))^{2/p}}{|f''(x)|^{2/p}}
\end{align*}
which is consistent with \eqref{eq:T:1:2:3:4:p:p}.

Lastly, for the $T_4$ term, we use bounds similar to \eqref{eq:T:4:1:in}, \eqref{eq:T:4:2:in}, and \eqref{eq:T:4:out}, combined with \eqref{eq:R1:p}--\eqref{eq:R2:p}, to deduce
\begin{align*}
|T_{4,1,in}|&\leq C |f''(x)| \left(\bar \eps B |f''(x)|  +\bar \eps^{(p+2)/p} B^{(p+2)/p} \frac{(\DD^p[f''](x,\alpha))^{2/p}}{|f''(x)|^{(p+2)/p}}  \right) \\
|T_{4,2,in}|&\leq C \bar \eps^{1/p} B^{(p+1)/p} |f''(x)|^{(p-1)/p} (\DD^p[f''](x))^{1/p} + C \bar \eps^{2/p} B^{(p+2)/p} \frac{(\DD^p[f''](x))^{2/p}}{|f''(x)|^{2/p}}\\
|T_{4,out}|&\leq  \frac{C B}{\bar \eps^2} |f''(x)|^2
\end{align*}
which concludes the proof of \eqref{eq:T:1:2:3:4:p:p}.

In order to prove \eqref{eq:T:1:2:3:4:p},  we now use \eqref{eq:T:1:2:3:4:p:p} in which we choose $\bar \eps$ depending on two cases:
\begin{align}
 \mbox{if} \
 \begin{cases}
 B^{2/p} (\DD^p[f''](x))^{2/p} \leq |f''(x)|^{2(p+1)/p},\\
 B^{2/p} (\DD^p[f''](x))^{2/p} \geq |f''(x)|^{2(p+1)/p},
 \end{cases}
 \  \mbox{then let} \
 \begin{cases}
 \bar \eps = 1,\\
 \bar \eps = |f''(x)| (B \DD^p[f''](x))^{-1/(p+1)}.
 \end{cases}
 \label{eq:bar:eps:choice:p}
 \end{align}
We thus obtain the desired estimate
\begin{align*}
|T_1(x)| +  |T_2(x)| + |T_3(x)| + T_4(x)|
&\leq C B \max \left\{ |f''(x)|^2, B^{2/(p+1)}  (\DD^p[f''](x))^{2/(p+1)}\right\}\notag\\
&\leq \frac{C \eps B^2 \DD^p[f''](x)}{|f''(x)|^{p-1}} + \frac{C B |f''(x)|^2}{\eps^{2/(p-1)}}
\end{align*}
where in the last inequality we have appealed to the $\eps$-Young inequality and the essential fact that $2< p+1$.

{\bf Estimate for the $T_5$ term.\ }
The first step is to show that
\begin{align}
|T_5(x)| &\leq C B \left( |\Lambda f'(x)| + \frac{1}{\bar \eps} |f''(x)| + \bar \eps^{(p+2)/p} B^{2/p} \frac{(\DD_p[f''](x))^{2/p}}{|f''(x)|^{(p+2)/p}}\right)
\label{eq:T:5:p:p}
\end{align}
holds when $p \in (1,2)$.
For this purpose, similarly to \eqref{eq:T:5:pv}, \eqref{eq:T:5:out}, and \eqref{eq:T:5:in}, by using \eqref{eq:R1:p}--\eqref{eq:R2:p} we arrive at
\begin{align*}
|T_{5,pv}| &\leq C B |\Lambda f'(x)|\\
|T_{5,in}| & \leq  C \bar \eps B |f''(x)| + C \bar \eps^{(p+2)/p} B^{(p+2)/p} \frac{(\DD^p[f''](x))^{2/p}}{|f''(x)|^{(p+2)/p}}\\
|T_{5,out}| &\leq  \frac{C B}{\bar \eps} |f''(x)|.
\end{align*}
Combining the above estimates proves \eqref{eq:T:5:p:p}. We now use \eqref{eq:T:5:p:p} in which we choose $\bar \eps$ precisely as in \eqref{eq:bar:eps:choice:p} to obtain the desired estimate
\begin{align*}
|T_5(x)|
&\leq C B \left( |\Lambda f'(x)| + |f''(x)| + B^{1/(p+1)} (\DD^p[f''](x))^{1/(p+1)}\right)\notag\\
&\leq C B \left(|\Lambda f'(x)| + \frac{1}{\eps^p} |f''(x)| + \frac{\eps B \DD^p[f''](x)}{|f''(x)|^p}\right)
\end{align*}
In the last inequality we have appealed to the $\eps$-Young inequality and the fact that $1< p+1$.
\end{proof}

\section{Proof of Theorem \ref{thm:cond}, part (i), Blowup Criterium for the Curvature}
\label{sec:proof:cond}

We evaluate \eqref{eq:f'':2} at a point $\bar x = \bar x(t)$ where the second derivative achieves its maximum, i.e. such that
\[
|f''(\bar x,t)| = \|f''(t)\|_{L^\infty} = M_\infty(t).
\]
At this point $\bar x$ we have
\[
\partial_x |f''(\bar x,t)|^2 = 0 \qquad \mbox{and} \qquad \LL_f [f''](\bar x,t) \geq 0.
\]
By Rademacher's theorem whose use is justified by the assumptions, we may then show that
\begin{align}
\frac{d}{dt} M_\infty(t)^2 = \frac{d}{dt} \| f''(t)\|_{L^\infty}^2
&= \partial_t |f''(\bar x,t)|^2  \notag\\
& \leq (\partial_t + v(\bar x,t) \partial_x + \LL_f) |f''(\bar x,t)|^2 \notag\\
&= 2 f''(\bar x,t) (T_1 + T_2 + T_3 + T_4) - \DD_f [f''](\bar x,t)
\label{eq:f'':3}
\end{align}
for almost every $t \in [0,T]$.

In order to complete the proof, we estimate the right hand side of \eqref{eq:f'':3}.
The dissipative term $\DD_f [f''](x)$ is non-negative, and is bounded from below explicitly  as in~\eqref{eq:Df:DD}, \eqref{eq:lower:bound}, and \eqref{eq:optimal:lower:bound}.
In order to estimate the $T_1,\ldots,T_4$ terms on the right side of \eqref{eq:f'':1}, we appeal to the upper bound \eqref{eq:T:1:2:3:4} obtained in Lemma~\ref{lem:T:1:2:3:4:5}. We thus obtain
\begin{align}\label{eq:M}
&\frac{d}{dt} M_\infty^2  + \frac{1}{2(1+B^2)}  \DD[f''](\bar x) + \frac 12 L_B(M_\infty)\leq  \frac{C_0 B}{\eps^2} M_\infty^3 + C_0 \eps B^2 \DD[f''](\bar x)
\end{align}
for some universal constant $C_0 \geq 1$.
We now choose the value of $\eps$ in Lemma~\ref{lem:T:1:2:3:4:5} as
\[
\eps = \min \left\{ \frac{1}{2 C_0 B^2(1+B^2)} , 1 \right\}
\]
in order to obtain
\[
\frac{d}{dt} M_\infty^2  + \frac 12 L_B(M_\infty) \leq K_{B}  M_\infty^3
\]
where
\[
K_{B} =  C_0  B \max \{1, 4 C_0^2 B^4(1+B^2)^2\}
\]
is a constant that depends solely on $B$.
To close the argument, we recall cf.~\eqref{eq:proof:better} that
\[
\lim_{M\to \infty} \frac{L_B(M)}{M^3} = \infty
\]
and that $L_B$ is continuous, which shows that there exists $M_\infty^* = M_\infty^*(B,C_0) $ such that
\[
\frac 12 L_B(M) \geq K_{B} M^3 \qquad \mbox{for any} \qquad M\geq M_\infty^*.
\]
Therefore, we obtain
\[
M_\infty(t) \leq \max\{ M_\infty(0) , M_\infty^* \}
\]
for all $t \in [0,T]$, which concludes the proof of the theorem.

\section{Proof of Theorem \ref{thm:local}, Local Existence}
\label{sec:local}
In view of the available  $L^p$ maximum principle for $f$, cf.~\eqref{eq:max:princ:Lp},  the proof consists of coupling the evolution of the maximal slope
\[
B(t) = \|f'(t)\|_{L^\infty}
\]
with that of the $L^p$ norm of $f''$
\[
M_p(t) = \|f''(t)\|_{L^p}.
\]
Our goal is to obtain an a priori estimate of the type
\[
\frac{d}{dt}(B^2 + M_p^2) \leq {\rm polynomial} (B^2 + M_p^2)
\]
from which the existence of solutions on a time interval that only depends on $B(0)^2 + M_p(0)^2$ follows by a standard approximation procedure.

For the evolution of $M_p(t)$ we split the proof in three cases:
\begin{enumerate}
\item $p= \infty$
\item $p \in [2,\infty)$
\item $p \in (1,2)$.
\end{enumerate}
We however note that in all of these three cases, but the $1D$ Sobolev embedding, as soon as $f''\in L^p$, we have $f' \in {C}^{(p-1)/p}$, and we have the bound
\[
[f']_{C^{(p-1)/p}} \leq \|f''\|_{L^p}
\]
for some positive constant $C$ that may depend on $p$. In particular, it follows from \eqref{eq:f':MOC:easy:bound} that the bound
\begin{align}
|\RSZ_1[f'](x,\alpha)| = |\Delta_\alpha f(x) - f'(x)| \leq |\alpha|^{(p-1)/p} \|f''\|_{L^p}
\label{eq:Lp:MOC}
\end{align}
for all $x,\alpha \in \RR$, and all $p \in (1,\infty]$. In view of this estimate, we immediately notice that the uniqueness of solutions in $W^{2,p}$  follows directly from Theorem~\ref{thm:uniqueness}, whose proof is given below.

\subsection{Evolution of the maximal slope $B(t)$}
We first multiply equation \eqref{eq:f'}  by $f'(x,t)$ and arrive at  the  equation
\begin{align}
(\partial_t + v \partial_x + \LL_f) |f'|^2 + \DD_f[f'] = T_6
\label{eq:f':square}
\end{align}
pointwise in  $(x,t)$,
where we have denoted
\[
T_6(x,t)=4 f'(x,t) \int_{\RR} \frac{ (\RSZ_1[f'](x,\alpha)) (\Delta_\alpha f(x,t)) }{(\Delta_\alpha f(x,t))^2 +1 } \frac{\delta_\alpha f'(x,t)}{(\delta_\alpha f(x))^2 + \alpha^2}d\alpha.
\]
Let $\bar x = \bar x(t)$ be a point such that $|f'(\bar x(t),t)| = B(t)$.
Using equation \eqref{eq:f':square} and Rademacher's theorem
we find that
\begin{align}
\frac{d}{dt}B^2(t)= \partial_t |f'(\bar x,t)|^2 =  T_6(\overline{x},t) - \left( \DD_f[f'](\bar x,t) + \LL_f|f'(\bar x,t)|^2\right)
\label{eq:B:evo:1}
\end{align}
for almost every $t$. Here we used that at a point of global maximum $\partial_x |f'|^2 = 0$. Note that also that at $\bar x$ we have $\LL_f |f'|^2 \geq 0$.
It remains to bound $T_6(\bar x,t)$. For this purpose we decompose $T_6$ as follows
\begin{align*}
T_6(x)&=
2 \int_{|\alpha|\leq \eps} \frac{ (\RSZ_1[f'](x,\alpha)) (\Delta_\alpha f(x,t)) }{(\Delta_\alpha f(x,t))^2 +1 } \frac{(\delta_\alpha f'(x,t))^2}{(\delta_\alpha f(x))^2 + \alpha^2}d\alpha \notag\\
&\quad +2 \int_{|\alpha|\leq \eps} \frac{ (\RSZ_1[f'](x,\alpha)) (\Delta_\alpha f(x,t)) }{(\Delta_\alpha f(x,t))^2 +1 } \frac{\delta_\alpha (f'(x,t))^2}{(\delta_\alpha f(x))^2 + \alpha^2}d\alpha\notag\\
&\quad + 4 f'(x,t) \int_{|\alpha|\geq \eps} \frac{ (\RSZ_1[f'](x,\alpha)) (\Delta_\alpha f(x,t)) }{((\Delta_\alpha f(x,t))^2 +1)^2} \frac{\delta_\alpha f'(x,t)}{\alpha^2}d\alpha\notag\\
&= T_{6,1,in}(x)+ T_{6,2,in}(x) + T_{6,out}(x)
\end{align*}
where $\eps = \eps(x)>0$ is to be determined, and we ignore the $t$-dependence of all factors.
Using \eqref{eq:Lp:MOC} and the fact that $\LL_f |f'(\bar x)|^2 \geq 0$ we may estimate
\begin{align*}
|T_{6,1,in}(\bar x)| &\leq \eps^{(p-1)/p} \|f''\|_{L^p} \DD_f[f'](\bar x)\\
|T_{6,2,in}(\bar x)| &\leq \eps^{(p-1)/p} \|f''\|_{L^p} \LL_f |f'(\bar x)|^2,
\end{align*}
while the bound \eqref{eq:f':bnd} yields
\begin{align*}
|T_{6,out}(\bar x)| \leq \frac{8 B^3}{\eps}.
\end{align*}
Letting
\[
\eps = \left( \frac{1}{2\|f''\|_{L^p}}\right)^{p/(p-1)}
\]
in the above three estimates, we arrive at
\[
|T_6(\bar x)| \leq \frac 12 \left( \DD_f[f'](\bar x) + \LL_f |f'(\bar x)|^2 \right) + C B^3 \|f''\|_{L^p}^{p/(p-1)}
\]
which combined with \eqref{eq:B:evo:1} gives
\begin{align}
\frac{d}{dt} B^2 \leq C B^3 \|f''\|_{L^p}^{p/(p-1)}
\label{eq:B:evo:final}
\end{align}
for some positive constant $C$ that may depend on $p \in (1,\infty]$.

\subsection{Case (i), $p=\infty$}

We recall the evolution of $M_\infty(t)$, cf.~\eqref{eq:M}, in which we take
\[
\eps = \eps(t) = \min \left\{ \frac{1}{2 C_0 B(t)^2 (1+B(t)^2)},1 \right\}
\]
to arrive at the a priori estimate
\begin{align}
&\frac{d}{dt} M_\infty^2  \leq   C_0 B \max \left\{2 C_0 B^2 (1+B^2),1 \right\} M_\infty^3
\label{eq:M:infty:evo}.
\end{align}
Combining \eqref{eq:B:evo:final} with \eqref{eq:M:infty:evo} we obtain that
\begin{align}
\frac{d}{dt} (B(t)^2 + M_\infty(t)^2)
&\leq C B(t)^3 M_\infty(t) + C B(t) (1+B(t)^2)^2 M_\infty(t)^3 \notag\\
&\leq C (1+ B(t)^2 + M_\infty(t)^2)^{4}
\label{eq:B:M:infty:evo}
\end{align}
for some positive constant $C$. Integrating \eqref{eq:B:M:infty:evo}, we obtain that there exists $T= T(\|f'_0\|_{L^\infty},\|f''_0\|_{L^\infty}) > 0$ on which the solution may be shown to exist and have finite $W^{2,\infty}$ norm.

\subsection{Case (ii), $2 \leq p < \infty$}
We consider the evolution of $|f''|^p$ in divergence form, cf.~\eqref{eq:f'':p:geq:2:div},  apply the upper bound given by Lemma~\ref{lem:T:1:2:3:4:5} for the terms $T_1,\ldots, T_5$ on the right side of \eqref{eq:f'':p:geq:2:div}, and the lower bound given in Lemma~\ref{lem:lower:fpp:Lp}, estimate \eqref{eq:lower:fpp:Lp}, and the bound given by \eqref{eq:Df:DD} for the dissipative term $\DD_f[f'']$ on the left side of \eqref{eq:f'':p:geq:2:div}, to arrive at
\begin{align}
&\left( \partial_t + \LL_f \right) |f''(x,t)|^p + \partial_x( v(x,t) |f''(x,t)|^p) + \frac{|f''(x,t)|^{p-2}}{2(1+B(t)^2)} \DD[f''](x,t) +  \frac{|f''(x,t)|^{2p}}{C_1 (1+B(t)^2) \|f''(t)\|_{L^p}^p} \notag\\
&\qquad \leq C_1 B(t) |f''(x,t)|^p |H f''(x,t)| + \frac{C_1 B(t)}{\eps(t)^2} |f''(x,t)|^{p+1} + C_1 \eps B(t)^2 |f''(x,t)|^{p-2} \DD[f''](x,t)
\label{eq:W:2:p:evo:1}
\end{align}
pointwise in $x$ and $t$. Choosing
\[
\eps(t) = \min\left\{ \frac{1}{4 C_1 B(t)^2 (1+ B(t)^2)}, 1 \right\}
\]
we conclude from \eqref{eq:W:2:p:evo:1} that
\begin{align}
&\left( \partial_t + \LL_f \right) |f''(x,t)|^p + \partial_x( v(x,t) |f''(x,t)|^p)  + \frac{|f''(x,t)|^{p-2}}{4(1+B(t)^2)} \DD[f''](x,t) + \frac{|f''(x,t)|^{2p}}{C_1 (1+B(t)^2) \|f''(t)\|_{L^p}^p} \notag\\
&\qquad \leq C_1 B(t) |f''(x,t)|^p |H f''(x,t)| + C B(t)^5 (1+B(t)^2) |f''(x,t)|^{p+1}
\label{eq:W:2:p:evo:2}
\end{align}
where $C = C(C_1)>0$ is a constant.

At this stage we integrate \eqref{eq:W:2:p:evo:2} for $x\in \RR$. First we note that
\begin{align}
\int_{\RR} \LL_f[|f''|^p](x) dx = P.V. \int_{\RR} \int_{\RR}  \frac{|f''(x)|^p - |f''(x-\alpha)|^p}{ (f(x) - f(x-\alpha))^2 + \alpha^2} d\alpha dx= 0.
\label{eq:LL:integrated:0}
\end{align}
This fact may be seen by changing variables $x \to x-\alpha$. We thus obtain an a priori estimate for the evolution of $M_p(t) = \|f''(t)\|_{L^p}$ as
\begin{align}
\frac{d}{dt} M_p(t)^p +  \frac{M_{2p}(t)^{2p}}{C_1 (1+B(t)^2) M_p(t)^p} \leq C B(t) (1+B(t)^2)^3 M_{p+1}(t)^{p+1}
\label{eq:W:2:p:evo:3}
\end{align}
by also using that the Hilbert transform is bounded on $L^p$. Furthermore, since for $p>1$ we have $p+1 < 2p$, we may interpolate
\begin{align}
M_{p+1}(t) \leq M_{p}(t)^{(p-1)/(p+1)} M_{2p}(t)^{2/(p+1)}
\label{eq:Lp:interpolation}
\end{align}
which combined with  \eqref{eq:W:2:p:evo:3} and the $\eps$-Young inequality, yields
\begin{align}
\frac{d}{dt} M_p(t)^p &+  \frac{M_{2p}(t)^{2p}}{C_1 (1+B(t)^2) M_p(t)^p} \notag\\
&\leq C B(t) (1+B(t)^2)^3 M_{p}(t)^{p-1} M_{2p}(t)^{2}\notag\\
&\leq \frac{M_{2p}(t)^{2p}}{2 C_1 (1+B(t)^2) M_p(t)^p} + C B(t)^{p/(p-1)} (1+B(t)^2)^{(3p+1)/(p-1)} M_p(t)^{p^2/(p-1)}
\label{eq:W:2:p:evo:4}
\end{align}
In conclusion, we obtain
\[
\frac{d}{dt} M_p(t)^2 \leq C B(t)^{p/(p-1)} (1+B(t)^2)^{(3p+1)/(p-1)} \left(M_p(t)^{2}\right)^{(3p-2)/(2p-2)}
\]
which combined with \eqref{eq:B:evo:final} gives
\begin{align}
\frac{d}{dt} \left(B(t)^2 + M_p(t)^2\right) \leq C (1 + B(t)^2 + M_p(t)^2)^{5p/(p-1)}
\label{eq:W:2:p:evo:5}
\end{align}
for some positive constant $C$. Integrating \eqref{eq:W:2:p:evo:5}, we obtain that there exists $T= T(\|f'_0\|_{L^\infty},\|f''_0\|_{L^p}) > 0$ on which the solution may be shown to exist and have finite $W^{2,p} \cap W^{1,\infty}$ norm.

\subsection{Case (iii), $1 < p < 2$}

We proceed similarly to the case $p\geq 2$, but instead of applying to \eqref{eq:f'':p:geq:2:div}, we use \eqref{eq:f'':p:leq:2:div}. For the dissipative term on the left side of \eqref{eq:f'':p:geq:2:div}, we use the minimum between the lower bounds provided by \eqref{eq:lower:fpp:Lp} and \eqref{eq:lower:fpp:Dp:Lp} in Lemma~\ref{lem:lower:fpp:Lp} and \eqref{eq:Df:DD}. For the nonlinear terms on the right side of \eqref{eq:f'':p:leq:2:div} we use the minimum between the upper bounds provided by Lemmas~\ref{lem:T:1:2:3:4:5} and~\ref{lem:T:1:2:3:4:5:p}. The resulting pointwise in $(x,t)$ inequality is
\begin{align}
&(\partial_t + \LL_f) |f''|^p  + \partial_x( v  |f''|^p )
+ \frac{1}{C_2(1+B^2)} \min\left\{ \frac{\DD[f'']}{|f''|^{2-p}} , \DD^p[f''] \right\} + \frac{|f''|^{2p}}{C_2(1+B^2) \|f''\|_{L^p}^p} \notag\\
&\quad \leq C_2  \eps B^2 \min\left\{\frac{\DD[f'']}{|f''|^{2-p}} , \DD^p[f'']\right\} + \frac{C_2 B}{\eps^{2/(p-1)}} |f''|^{p+1} + C_2 B |H f''|\, |f''|^p
\label{eq:W:2:p:evo:6}
\end{align}
for some constant $C_2>0$.  Choosing
\[
\eps(t) = \min\left\{ \frac{1}{2 C_2^2 B(t)^2 (1+ B(t)^2)}, 1\right\}
\]
we conclude from \eqref{eq:W:2:p:evo:6} that
\begin{align}
&(\partial_t + \LL_f) |f''|^p  + \partial_x( v  |f''|^p )
+ \frac{1}{2 C_2(1+B^2)} \min\left\{ \frac{\DD[f'']}{|f''|^{2-p}} , \DD^p[f''] \right\} + \frac{|f''|^{2p}}{C_2(1+B^2) \|f''\|_{L^p}^p} \notag\\
&\quad \leq C (1+B^2)^{(p+7)/(2p-2)}  (|f''|^{p+1} +  |H f''|\, |f''|^p).
\label{eq:W:2:p:evo:7}
\end{align}
Upon integrating \eqref{eq:W:2:p:evo:7} for $x\in \RR$, using the identity \eqref{eq:LL:integrated:0}, the boundedess of  $H$ on $L^p$, and the interpolation bound \eqref{eq:Lp:interpolation}, we thus arrive at
\begin{align}
\frac{d}{dt} M_p(t)^p + \frac{M_{2p}(t)^{2p}}{C_2 (1+B(t)^2) M_p(t)^p}
&\leq C (1+B(t)^2)^{(p+7)/(2p-2)} M_{p+1}(t)^{p+1} \notag\\
&\leq C (1+B(t)^2)^{(p+7)/(2p-2)} M_{p}(t)^{p-1} M_{2p}(t)^2.
\label{eq:W:2:p:evo:8}
\end{align}
Similarly to \eqref{eq:W:2:p:evo:4}--\eqref{eq:W:2:p:evo:5}, since $p+1>2$, and by using the $\eps$-Young inequality, we conclude from \eqref{eq:W:2:p:evo:8} that
\begin{align}
\frac{d}{dt} M_p(t)^2 \leq C (1+B(t)^2)^{(p^2+9p-2)/(2 (p-1)^2)} (M_p(t)^2)^{(3p-2)/(2p-2)}.
\label{eq:W:2:p:evo:9}
\end{align}
Combining with \eqref{eq:B:evo:final} we finally arrive at
\[
\frac{d}{dt} (B(t)^2 + M_p(t)^2) \leq C (1 + B(t)^2 + M_p(t)^2)^{2p(p+1)/(p-1)^2}.
\]
Integrating \eqref{eq:W:2:p:evo:5}, we obtain that there exists $T= T(\|f'_0\|_{L^\infty},\|f''_0\|_{L^p}) > 0$ on which the solution may be shown to exist and have finite $W^{2,p} \cap W^{1,\infty}$ norm.


\section{Proof of Theorem~\ref{thm:small}, Global Existence for Small Datum}
\label{sec:proof:small}

The proof follows closely the estimates in Section~\ref{sec:local}. The major difference is that
assumption \eqref{eq:small:Lip:cond} and the maximum principle for $f'$ established in~\cite[Section 5]{CordobaGancedo09} show that
 \begin{align}
\|f'(t)\|_{L^\infty} \leq B \leq \frac{1}{C_*}
\label{eq:Lip:small}
\end{align}
for all $t>0$.
Thus, we do not need to consider the evolution of $B(t)$, as we have
\[
B(t) \leq \frac{1}{C_*}
\]
for $t\in[0,T)$, where $T>0$ is the maximal existence time in $W^{2,p}$. For simplicity, we split the proof in three cases based on the value of $p\in(1,\infty]$:
\begin{enumerate}
\item $p= \infty$
\item $p \in [2,\infty)$
\item $p \in (1,2)$.
\end{enumerate}

\subsection{Case (i), $p=\infty$} We use the estimate \eqref{eq:M}, but here we apply lower bound \eqref{eq:lower:bound} instead of \eqref{eq:optimal:lower:bound:*}, and we set $\eps=1$. We arrive at
\[
\frac{d}{dt} M_\infty^2  + \frac{1}{2(1+B^2)}  \DD[f''](\bar x) + \frac{M_{\infty}^3}{48 B (1+B^2)} \leq  C_0 B M_\infty^3 + C_0 B^2 \DD[f''](\bar x)
\]
where $\bar x = \bar x(t)$ is a point at which $M_\infty = |f''(\bar x,t)|$.
For $B$ small enough, so that
\[
2 C_0 B^2 (1+B^2)\leq 1 \qquad \mbox{and} \qquad 100 C_0 B^2 (1+B^2) \leq 1
\]
holds, we thus  obtain
\[
\frac{d}{dt} M_\infty +\frac{1}{50B(1+B^2)}M_\infty^2 \leq 0.
\]
Integrating the above ODE we obtain that
\[
M_\infty(t) \leq  \frac{M_\infty(0)}{1 +  \frac{M_\infty(0)}{100 B} t}
\]
for all $t\geq 0$, which proves \eqref{eq:curvature:decay}.

\subsection{Case (ii), $2 \leq p < \infty$}
We use the first line of estimate \eqref{eq:W:2:p:evo:2}, but instead of using \eqref{eq:lower:fpp:Lp} to bound the dissipative term $\DD[f'']$ from below, we appeal to \eqref{eq:lower:bound}. We arrive at
\begin{align*}
&\left( \partial_t + \LL_f \right) |f''(x,t)|^p + \partial_x( v(x,t) |f''(x,t)|^p)  + \frac{|f''(x,t)|^{p+1}}{96 B(1+B^2)} + \frac{|f''(x,t)|^{p-2} \DD[f''](x,t)}{4 (1+B^2)}\notag\\
&\qquad \leq C_1 B |f''(x,t)|^p |H f''(x,t)| + C B^5 (1+B^2) |f''(x,t)|^{p+1}.
\end{align*}
Integrating the above over $x\in \RR$, similarly to \eqref{eq:W:2:p:evo:3} we obtain
\begin{align}
\frac{d}{dt} M_p(t)^p + \frac{M_{p+1}(t)^{p+1}}{96 B(1+B^2)} + \frac{1}{4 (1+B^2)} \int_{\RR} |f''(x,t)|^{p-2} \DD[f''](x,t) dx \leq C B (1+B^2)^3 M_{p+1}(t)^{p+1}
\label{eq:L2:new}
\end{align}
for some $C>0$. If $B$ is chosen small enough so that
\begin{align*}
200 C B^2 (1+B^2)^4 \leq 1
\end{align*}
we thus obtain
\begin{align*}
\frac{d}{dt} M_p(t)^p + \frac{M_{p+1}(t)^{p+1}}{200 B(1+B^2)}  + \frac{1}{4 (1+B^2)} \int_{\RR} |f''(x,t)|^{p-2} \DD[f''](x,t) dx \leq 0.
\end{align*}
Integrating the above in time and noting that
\begin{align*}
\int_{\RR} |f''(x,t)|^{p-2} \DD[f''](x,t) dx
&= \frac{1}{2} \int_{\RR}\int_{\RR} \frac{(|f''(x)|^{p-2} + |f''(x-\alpha)|^{p-2})(f''(x)-f''(x-\alpha))^2}{\alpha^2} d\alpha dx \notag\\
&\geq \frac{4}{p^2} \int_{\RR}\int_{\RR} \frac{(|f''(x)|^{p/2} - |f''(x-\alpha)|^{p/2})^2}{\alpha^2} d\alpha dx \notag\\
&= \frac{4}{p^2} \| |f''|^{p/2}\|_{\dot{H}^{1/2}}
\end{align*}
concludes the proof of \eqref{eq:mpfpp}.

\subsection{Case (iii), $1< p< 2$}
We use estimate \eqref{eq:W:2:p:evo:7} in which we bound from below the dissipative terms from below by appealing to Lemma~\ref{lem:lower:Lip}, and arrive at
\begin{align*}
&(\partial_t + \LL_f) |f''|^p  + \partial_x( v  |f''|^p )
 +  \frac{ |f''|^{p+1}}{200 B C_2(1+B^2)}+ \frac{|f''|^{2p}}{C_2(1+B^2) \|f''\|_{L^p}^p} \notag\\
&\quad \leq C (1+B^2)^{(p+7)/(2p-2)}  (|f''|^{p+1} +  |H f''|\, |f''|^p).
\end{align*}
Integrating the above over $x \in \RR$, and using that $H$ is bounded on $L^p$ in this range of $p$, we thus arrive at
\begin{align*}
\frac{d}{dt} M_p(t)^p +  \frac{M_{p+1}(t)^{p+1}}{200 B C_2(1+B^2)}+ \frac{M_{2p}(t)^{2p}}{C_2(1+B^2) \|f''\|_{L^p}^p} \leq C (1+B^2)^{(p+7)/(2p-2)} M_{p+1}(t)^{p+1}.
\end{align*}
Lastly, choosing $B$ small enough so that
\begin{align*}
400 C_2 C B (1+B^2)^{(3p+5)/(2p-2)} \leq 1,
\end{align*}
we arrive at
\begin{align*}
\frac{d}{dt} M_p(t)^p +  \frac{M_{p+1}(t)^{p+1}}{400 B C_2(1+B^2)} \leq 0
\end{align*}
which upon integration in time concludes the proof of \eqref{eq:mpfpp:<2} and thus of the theorem.

\section{Proof of Theorem \ref{thm:cond}, part (ii), Blowup Criterium for Smooth Solutions}

We shall study the evolution of the $\|f\|_{H^k}(t)$ norms for $k\geq 3$. We show that they can be controlled by $\sup_{[0,T]}\|f'(t)\|_{L^\infty}$ and $\sup_{[0,T]}\|f''(t)\|_{L^\infty}$. Then Theorem \ref{thm:cond} concludes the proof. In fact, the $H^k$ norm of a solution with $k>3$ can be controlled already by $H^3$-norm as shown in \cite{ConstantinCordobaGancedoRodriguezStrain13},  Section 5.2. Therefore we may assume $k=3$. We start by dealing with the evolution of $\|f''(t)\|_{L^2}^2$. We use inequality \eqref{eq:L2:new} with $\eps$ small enough and $p=2$, to obtain
\begin{align*}
\frac{d}{dt} \|f''\|_{L^2}^2 +\frac{\|f''\|_{\dot{H}^{1/2}}^2}{2C_0 (1+B^2)}
&\leq C(B)\|f''\|_{L^3}^3 \leq C(B)\|f''\|_{L^\infty} \|f''\|_{L^2}^2,
\end{align*}
and therefore
\begin{equation}\label{eq:boundL2pp}
\|f''(t)\|_{L^2}^2\leq \|f''_0\|_{L^2}^2\exp \left(C(B)\int_0^t\|f''(s)\|_{L^\infty}ds\right).
\end{equation}
We use equation \eqref{eq:muskat} to split
$$
\int_{\RR}f'''(x)f'''_t(x)dx=I_1+I_2+I_3+I_4,
$$
where
$$
I_1=\int_{\RR}f'''(x)\int_{\RR} (f''''(x)\alpha-\delta_\alpha f'''(x)) \left( \frac{1}{ (\delta_\alpha f(x))^2 + \alpha^2} \right)  d\alpha dx,
$$
$$
I_2=3\int_{\RR}f'''(x)\int_{\RR}(f'''(x)\alpha-\delta_\alpha f''(x))\partial_x \left( \frac{1}{ (\delta_\alpha f(x))^2 + \alpha^2} \right) d\alpha dx,
$$
$$
I_3=3\int_{\RR}f'''(x)\int_{\RR}(f''(x)\alpha-\delta_\alpha f'(x))\partial_x^2  \left( \frac{1}{ (\delta_\alpha f(x))^2 + \alpha^2} \right)  d\alpha dx,
$$
$$
I_4=\int_{\RR}f'''(x)\int_{\RR}(f'(x)\alpha-\delta_\alpha f(x))\partial_x^3 \left( \frac{1}{ (\delta_\alpha f(x))^2 + \alpha^2} \right)  d\alpha dx.
$$
In $I_1$ it is possible to decompose further and obtain
\[
I_1=\int_{\RR}f'''(x)f''''(x)P.V.\int_{\RR}\frac{\alpha}{(\delta_\alpha f(x))^2 + \alpha^2}d\alpha dx -\frac12\int_{\RR}\DD_f[f'''](x)dx:=I_{1,1}+I_{1,2}.
\]
We bound $I_{1,1}$ as
\begin{align*}
I_{1,1}=&\int_{\RR}|f'''(x)|^2P.V.\int_{\RR}\frac{\alpha\delta_\alpha f(x)\delta_\alpha f'(x)}{((\delta_\alpha f(x))^2 + \alpha^2)^2}d\alpha\\
=&\int_{\RR}|f'''(x)|^2P.V.\int_{\RR}\frac1\al\Big(\frac{\Delta_\alpha f(x)}{((\Delta_\alpha f(x))^2 +1)^2}-\frac{f'(x)}{((f'(x))^2+1)^2}\Big)\Delta_\alpha f'(x)d\alpha dx\\
&\quad+\int_{\RR}|f'''(x)|^2\frac{f'(x)\Lambda f'(x)}{((f'(x))^2+1)^2}dx:=I_{1,1,1}+I_{1,1,2}.
\end{align*}
In $I_{1,1,1}$ one finds extra cancelation in such a way that splitting in the regions $|\al|<r$ and $|\al|>r$ and optimizing in $r$, it is possible to obtain as before
$$
I_{1,1,1}\leq C(B)\|f''\|_{L^\infty}\|f'''\|_{L^2}^2.
$$
For $I_{1,1,2}$, the Gagliardo-Nirenberg interpolation inequality $\|g\|_{L^4}\leq C \|g\|_{L^2}^{1/2}\|g\|_{\dot{H}^{1/2}}^{1/2}$  yields
\begin{align*}
I_{1,1,2}&\leq \|f''\|_{L^2}\|f'''\|_{L^4}^2
\leq C \|f''\|_{L^2}\|f'''\|_{L^2}\|f'''\|_{\dot{H}^{1/2}}\\
&\leq C(B)\|f''\|^2_{L^2}\|f'''\|^2_{L^2}+ \frac{\|f'''\|_{\dot{^H}^{1/2}}^2}{32(1+B^2)}
\end{align*}
by using  the $\eps$-Young inequality. This yields
$$
I_{1,1}=I_{1,1,1}+I_{1,1,2}\leq C(B)(\|f''\|_{L^\infty}+\|f''\|^2_{L^2})\|f'''\|^2_{L^2}+\frac{\|f'''\|_{\dot{^H}^{1/2}}^2}{32(1+B^2)}.
$$
Using \eqref{eq:Df:DD}, \eqref{eq:H:0.5} together with \eqref{eq:lower:fppp} in $I_{1,2}$ we arrive at
$$
I_{1,2}\leq -\frac{1}{4(1+B^2)}\|f'''\|^2_{\dot{^H}^{1/2}}-\frac{1}{2^7(1+B^2)^3\|f''\|_{L^\infty}}\|f'''\|^3_{L^3}.
$$
Adding the last two estimates it is possible to obtain
\begin{equation}\label{eq:i1}
I_{1}=I_{1,1}+I_{1,2}\leq C(B)(\|f''\|_{L^\infty}+\|f''\|^2_{L^2})\|f'''\|^2_{L^2}-\frac{7\|f'''\|_{\dot{^H}^{1/2}}^2}{32(1+B^2)}
-\frac{\|f'''\|^3_{L^3}}{2^7(1+B^2)^3\|f''\|_{L^\infty}}.
\end{equation}
We are done with $I_{1}$. For $I_2$ we rewrite as
\begin{align*}
I_2=&6\int_{\RR}f'''(x)\int_{\RR}\frac{\Delta_\alpha f''(x)-f'''(x)}{\al}\frac{\Delta_\alpha f(x)}{(\Delta_\alpha f(x))^2 + 1)^2}\Delta_\al f'(x) d\alpha dx\\
=&6\int_{\RR}f'''(x)\int_{|\al|<\|f''\|_{L^\infty}^{-1}}\frac{\RSZ_1[f'''](x,\al)}{\al}\frac{\Delta_\alpha f(x)}{(\Delta_\alpha f(x))^2 + 1)^2}\Delta_\al f'(x) d\alpha dx\\
&+6\int_{\RR}f'''(x)\int_{|\al|>\|f''\|_{L^\infty}^{-1}}\frac{\RSZ_1[f'''](x,\al)}{\al}\frac{\Delta_\alpha f(x)}{(\Delta_\alpha f(x))^2 + 1)^2}\Delta_\al f'(x) d\alpha dx\\
:=&I_{2,in}+I_{2,out}.
\end{align*}
Inequality \eqref{eq:R1} allows us to get
\begin{align*}
I_{2,in}&\leq 6\|f''\|_{L^\infty}\int_{\RR}|f'''(x)|\DD[f](x)^{1/2}\int_{|\al|<\|f''\|_{L^\infty}^{-1}} \frac{d\al}{|\al|^{1/2}} dx \notag\\
&\leq C\|f''\|_{L^\infty}^{1/2}\int_{\RR}|f'''(x)|\DD[f](x)^{1/2}dx\\
&\leq C(B)\|f''\|_{L^\infty}\|f'''\|_{L^2}^2+\frac{\|f'''\|^2_{\dot{H}^{1/2}}}{32(1+B^2)},
\end{align*}
and
\begin{align*}
I_{2,out}
&\leq 12B\int_{\RR}|f'''(x)|\DD[f](x)^{1/2}\int_{|\al|>\|f''\|_{L^\infty}^{-1}} \frac{d\al}{|\al|^{3/2}}dx \notag\\
&\leq C(B)\|f''\|_{L^\infty}^{1/2}\int_{\RR}|f'''(x)|\DD[f](x)^{1/2}dx\\
&\leq C(B)\|f''\|_{L^\infty}\|f'''\|_{L^2}^2+\frac{\|f'''\|^2_{\dot{H}^{1/2}}}{32(1+B^2)}.
\end{align*}
These last to inequalities give the appropriate bound for $I_2$ such that adding \eqref{eq:i1} we obtain
\begin{equation}\label{eq:i1+i2}
I_{1}+I_2\leq C(B)(\|f''\|_{L^\infty}+\|f''\|^2_{L^2})\|f'''\|^2_{L^2}-\frac{5\|f'''\|_{\dot{^H}^{1/2}}^2}{32(1+B^2)}
-\frac{\|f'''\|^3_{L^3}}{2^7(1+B^2)^3\|f''\|_{L^\infty}}.
\end{equation}
It is possible to decompose further in $I_3$ as follows
\begin{align*}
I_3=&c_{3,1}\int_{\RR}f'''(x)\int_{\RR}(f''(x)\alpha-\delta_\alpha f'(x))\frac{(\delta_\alpha f'(x))^2}{((\delta_\alpha f(x))^2 + \alpha^2)^2}d\alpha dx\\
&+c_{3,2}\int_{\RR}f'''(x)\int_{\RR}(f''(x)\alpha-\delta_\alpha f'(x))\frac{(\delta_\alpha f(x))^2(\delta_\alpha f'(x))^2}{((\delta_\alpha f(x))^2 + \alpha^2)^3}d\alpha dx\\
&+c_{3,3}\int_{\RR}f'''(x)\int_{\RR}(f''(x)\alpha-\delta_\alpha f'(x))\frac{\delta_\alpha f(x)\delta_\alpha f''(x)}{((\delta_\alpha f(x))^2 + \alpha^2)^2}d\alpha dx\\
:=&I_{3,1}+I_{3,2}+I_{3,3}.
\end{align*}
The identity
\begin{equation}\label{eq:taylor3}
f''(x)\alpha-\delta_\alpha f'(x)=\al^2\int_0^1r f'''(x+(r-1)\al)dr,
\end{equation}
allows us to get
\begin{align*}
I_{3,1}\leq& C\|f''\|_{L^\infty}^2\int_0^1\!\int_{|\al|<B\|f''\|_{L^\infty}^{-1}} \int_{\RR} (|f'''(x)|^2+|f'''(x+(r-1)\al)|^2) dxd\alpha  dr\\
& +CB^2\int_0^1\!\int_{|\al|>B\|f''\|_{L^\infty}^{-1}}  \int_{\RR} (|f'''(x)|^2+|f'''(x+(r-1)\al)|^2) dx \frac{d\alpha}{|\al|^{2}} dr\\
\leq& C(B)\|f''\|_{L^\infty}\|f'''\|^2_{L^2}.
\end{align*}
An analogous approach for $I_{3,2}$ gives
\begin{align*}
I_{3,2}\leq&  C(B)\|f''\|_{L^\infty}\|f'''\|^2_{L^2}.
\end{align*}
For the $I_{3,3}$ term, we decompose further:
\begin{align*}
I_{3,3}=&c_{3,3}\int_{\RR}f'''(x)\int_{|\al|<\zeta}\frac{f''(x)-\Delta_\alpha f'(x)}{\al}\frac{\Delta_\alpha f(x)}{(\Delta_\alpha f(x))^2 + 1)^2}\Delta_\alpha f''(x)d\alpha dx\\
&+c_{3,3}\int_{\RR}f'''(x)\int_{|\al|>\zeta}\frac{1}{\al^2}(f''(x)-\Delta_\alpha f'(x))\frac{\Delta_\alpha f(x)}{((\Delta_\alpha f(x))^2 + 1)^2}\delta_\alpha f''(x)d\alpha dx\\
:=&I_{3,3,in}+I_{3,3,out}.
\end{align*}
For the outer term, inequality
\begin{equation}\label{eq:interL4}
\|f''\|_{L^4}^2\leq C\|f'\|_{L^\infty}\|f'''\|_{L^2}
\end{equation} yields
$$
I_{3,3,out}\leq \frac{C}{\zeta}\|f''\|_{L^4}^2\|f'''\|_{L^2}\leq \frac{CB }{\zeta}\|f'''\|_{L^2}^2 .
$$
Identity \eqref{eq:taylor3} together with
\begin{equation}\label{eq:taylor3o1}
\Delta_\al f''(x)=\int_0^1f'''(x+(s-1)\al)ds,
\end{equation} allow us to obtain
for the inner term
\begin{align*}
I_{3,3,out}
&\leq c_{3,3}\int_0^1\!\!\int_0^1\!\!\int_{|\al|<\zeta}\!\int_{\RR}|f'''(x)||f'''(x+(r-1)\al)||f'''(x+(s-1)\al)|dx d\al ds dr \notag\\
&\leq 2c_{3,3}\zeta \|f'''\|_{L^3}^3.
\end{align*}
We take
\[
\zeta=\frac{1}{2^8(1+B^2)^3c_{3,3}\|f''\|_{L^\infty}}
\]
to find
$$
I_3=I_{3,1}+I_{3,2}+I_{3,3}\leq C(B)\|f''\|_{L^\infty}\|f'''\|_{L^2}^2+\frac{\|f'''\|_{L^3}^3}{2^7(1+B^2)^3\|f''\|_{L^\infty}}.
$$
After comparing $I_3$ with $I_1+I_2$ in \eqref{eq:i1+i2} we obtain
\begin{equation}\label{eq:i1+i2+i3}
I_{1}+I_2+I_3\leq C(B)(\|f''\|_{L^\infty}+\|f''\|^2_{L^2})\|f'''\|^2_{L^2}- \frac{5 \|f'''\|_{\dot{^H}^{1/2}}^2}{32(1+B^2)}.
\end{equation}
For the last term, we decompose using Leibniz rule to find
$$
I_4=I_{4,1}+I_{4,2}+I_{4,3}+I_{4,4}+I_{4,5}
$$
where
\begin{align*}
I_{4,1}:=&c_{4,1}\int_{\RR}f'''(x)\int_{\RR}(f'(x)\alpha-\delta_\alpha f(x))\frac{\delta_\alpha f(x)(\delta_\alpha f'(x))^3}{((\delta_\alpha f(x))^2 + \alpha^2)^3}d\alpha dx,
\end{align*}
\begin{align*}
I_{4,2}:=&c_{4,2}\int_{\RR}f'''(x)\int_{\RR}(f'(x)\alpha-\delta_\alpha f(x))\frac{(\delta_\alpha f(x))^3(\delta_\alpha f'(x))^3}{((\delta_\alpha f(x))^2 + \alpha^2)^4}d\alpha dx,
\end{align*}\begin{align*}
I_{4,3}:=&c_{4,3}\int_{\RR}f'''(x)\int_{\RR}(f'(x)\alpha-\delta_\alpha f(x))\frac{\delta_\alpha f'(x)\delta_\alpha f''(x)}{((\delta_\alpha f(x))^2 + \alpha^2)^2}d\alpha dx,
\end{align*}\begin{align*}
I_{4,4}:=&c_{4,4}\int_{\RR}f'''(x)\int_{\RR}(f'(x)\alpha-\delta_\alpha f(x))\frac{(\delta_\al f(x))^2\delta_\alpha f'(x)\delta_\alpha f''(x)}{((\delta_\alpha f(x))^2 + \alpha^2)^3}d\alpha dx,
\end{align*}\begin{align*}
I_{4,5}:=&c_{4,5}\int_{\RR}f'''(x)\int_{\RR}(f'(x)\alpha-\delta_\alpha f(x))\frac{\delta_\alpha f(x)\delta_\alpha f'''(x)}{((\delta_\alpha f(x))^2 + \alpha^2)^2}d\alpha dx.
\end{align*}
In $I_{4,1}$ we bound as follows
\begin{align*}
I_{4,1}\leq& C\int_{|\al|<\|f''\|_{L^\infty}^{-1}}d\alpha \|f'''\|_{L^2}\|f''\|_{L^8}^4+C\|f'\|_{L^\infty}^2\int_{|\al|>\|f''\|_{L^\infty}^{-1}} \|f'''\|_{L^2}\|f''\|_{L^4}^2
 \frac{d\alpha}{|\al|^{2}} \\ \leq& C(B)\|f''\|_{L^\infty}\|f'''\|_{L^2}^2,
\end{align*}
where the last inequality is given using interpolation inequality $\|f''\|_{L^8}^4\leq \|f'\|_{L^\infty}\|f''\|_{L^\infty}^2\|f'''\|_{L^2}$ together with \eqref{eq:interL4}.
The same approach allows us to conclude for $I_{4,2}$ that
\begin{align*}
I_{4,2}\leq& C(B)\|f''\|_{L^\infty}\|f'''\|_{L^2}^2.
\end{align*}
Using identity \eqref{eq:taylor3o1} in $I_{4,3}$ we arrive at
\begin{align*}
I_{4,3}\leq& C\|f''\|_{L^\infty}^2 \|f'''\|_{L^2}^2 \int_0^1\! \int_{|\al|<\|f''\|_{L^\infty}^{-1}}d\alpha ds +C\|f'\|_{L^\infty}^2 \|f'''\|_{L^2}^2 \int_0^1\!\int_{|\al|>\|f''\|_{L^\infty}^{-1}} \frac{d\alpha}{|\al|^{2}} ds \\
\leq& C(B)\|f''\|_{L^\infty}\|f'''\|_{L^2}^2.
\end{align*}
The same procedure yields
\begin{align*}
I_{4,4}\leq C(B)\|f''\|_{L^\infty}\|f'''\|_{L^2}^2.
\end{align*}
Finally, in $I_{4,5}$ we use next splitting
\begin{align*}
I_{4,5}\leq& C\|f''\|_{L^\infty}\int_{\RR}|f'''(x)|\Big(\int_{|\al|<B\|f''\|_{L^\infty}^{-1}}d\al\Big)^{1/2}\Big(\int_{|\al|<B\|f''\|_{L^\infty}^{-1}}|\Delta_\alpha f'''(x)|^2d\alpha\Big)^{1/2} dx\\
&+C\|f'\|_{L^\infty}\int_{|\al|>B\|f''\|_{L^\infty}^{-1}} \int_{\RR}(|f'''(x)|^2+|f'''(x-\al)|^2) dx \frac{d\alpha}{|\al|^{2}}\\
\leq& C(B)\|f''\|_{L^\infty}^{1/2}\int_{\RR}|f'''(x)|(\DD(f''')(x))^{1/2}dx+C(B)\|f''\|_{L^\infty}\|f'''\|_{L^2}^2\\
\leq& C(B)\|f''\|_{L^\infty}\|f'''\|_{L^2}^2+ \frac{\|f'''\|_{\dot{H}^{1/2}}}{32(1+B^2)}.
\end{align*}
Above estimates allow us to conclude that
$$
I_{4}=\sum_{k=1}^5 I_{4,k}\leq C(B)\|f''\|_{L^\infty}\|f'''\|_{L^2}^2+ \frac{\|f'''\|_{\dot{H}^{1/2}}}{32(1+B^2)}.
$$
Adding to \eqref{eq:i1+i2+i3} we obtain that
$$
\frac{d}{dt}\|f'''\|_{L^2}^2+ \frac{\|f'''\|_{\dot{H}^{1/2}}}{8(1+B^2)}\leq C(B)(\|f''\|_{L^2}^2+\|f''\|_{L^\infty})\|f'''\|_{L^2}^2.
$$
The bound for $\|f''\|_{L^2}$ in \eqref{eq:boundL2pp}, the control of $\|f''\|_{L^\infty}$ and integration in time yield the desired result.

\section{Proof of Theorem \ref{thm:uniqueness}, Uniqueness}
We consider two Muskat solutions $f_1$ and $f_2$ satisfying the hypothesis of the theorem with the same initial data $f_0(x)$. From \eqref{eq:muskat} and a small computation we obtain the equation for the difference $g = f_1-f_2$,
\begin{align}
\partial_t g + v_1 \partial_x g + \LL_{f_1}[g] = T_7
\label{eq:g:eqn}
\end{align}
where $v_1(x,t)$ is as defined in \eqref{eq:v:def} in terms of $f_1(x,t)$, $\LL_{f_1}$ is defined as in \eqref{eq:LL:f:def}, and
\begin{align*}
T_7(x,t) = P.V. \int_{\RR} \frac{\delta_\alpha g(x,t)}{\alpha^2} \frac{\left( f_2'(x,t) - \Delta_\alpha f_2(x,t) \right) \left( \Delta_\alpha f_1(x,t) + \Delta_\alpha f_2(x,t) \right)}{(1 + (\Delta_\alpha f_1(x,t))^2)(1 + (\Delta_\alpha f_2(x,t))^2)} d\alpha.
\end{align*}
Let $B = \sup_{t\in [0,T],j=1,2} \| f'_j\|_\infty$. By assumption, $f_2'$ has a uniform modulus of continuity $\rho$ on $[0,T]$, and thus by \eqref{eq:f':MOC:easy:bound} we may find an $\eps = \eps(B,\rho) > 0$ such that
\begin{align}
| f_2'(x,t)- \Delta_\alpha f_2(x,t)| \leq \rho(|\alpha|) \leq \rho(\eps) \leq \frac{1}{2}
\label{eq:f2:MOC}
\end{align}
for all $|\alpha|\leq \eps$, and all $(x,t) \in \RR \times [0,T]$.
We fix this value of $\eps$ throughout the rest of the proof.
Denote
\begin{align*}
K_{1,2}(x,\alpha,t) = \frac{\left( f_2'(x,t) - \Delta_\alpha f_2(x,t) \right) \left( \Delta_\alpha f_1(x,t) + \Delta_\alpha f_2(x,t) \right)}{(1 + (\Delta_\alpha f_1(x,t))^2)(1 + (\Delta_\alpha f_2(x,t))^2)}.
\end{align*}
It follows from \eqref{eq:f2:MOC} that
\begin{align}
|K_{1,2}(x,\alpha,t)| \leq \frac{1}{2}
\label{eq:K12:in}
\end{align}
for all $|\alpha|\leq\eps$, while
the Lipschitz assumption on $f_1$ and $f_2$ directly yields
\begin{align}
|K_{1,2}(x,\alpha,t)| \leq 2B
\label{eq:K12:out}
\end{align}
for all $|\alpha|\geq \eps$, uniformly in $x$ and $t$.

Upon multiplying \eqref{eq:g:eqn} by $g(x,t)$, and recalling the definition \eqref{eq:DD:f:def}, we obtain
\begin{align}
& (\partial_t + v_1 \partial_x + \LL_{f_1}) |g|^2 + \DD_{f_1}[g] \notag\\
&\quad = P.V. \int_{|\alpha|\leq \eps} \frac{\delta_\alpha (g^2(x))}{\alpha^2} K_{1,2}(x,\alpha) d\alpha+ P.V. \int_{|\alpha|\leq \eps} \frac{(\delta_\alpha g(x))^2}{\alpha^2} K_{1,2}(x,\alpha) d\alpha \notag\\
&\quad \qquad + 2 g(x,t) P.V. \int_{|\alpha|\geq \eps} \frac{\delta_\alpha g(x)}{\alpha^2} K_{1,2}(x,\alpha) d\alpha \notag\\
&\quad =: T_{7,1,in} + T_{7,2,in} + T_{7,out}.
\label{eq:g:eqn:2}
\end{align}
First, we notice that in view of \eqref{eq:K12:in} we have
\[
|T_{7,2,in}| \leq \frac 12  \DD_{f_1}[g],
\]
while in view of \eqref{eq:K12:out} and the Cauchy-Schwartz inequality,
we have
\[
|T_{7,out}| \leq 4B |g| (\DD_{f_1}[g])^{1/2} \eps^{-1/2} \leq \frac 12 \DD_{f_1}[g] + 2 B^2 \eps^{-1} |g|^2.
\]
The above two inequalities combined with \eqref{eq:g:eqn:2} yield that
\begin{align}
& (\partial_t + v_1 \partial_x + \LL_{f_1}) |g|^2 \leq 2 B^2 \eps^{-1} |g|^2 + T_{7,1,in}.
\label{eq:g:eqn:3}
\end{align}
To conclude, we note that the decay assumptions of the theorem guarantee that there exists a point, denoted by $\overline{x}=\overline{x}(t)$, where $|g(\overline{x},t)|=\|g(t)\|_{L^\infty}$. At this point of global maximum we have that $\partial_x |g|^2  =0$, and $\LL_{f_1}|g|^2 \geq 0$. Moreover,
\[
|T_{7,1,in}(\bar x(t),t)| \leq \frac 12 (\LL_{f_1}|g|^2) (\bar x(t),t)
\]
and thus from \eqref{eq:g:eqn:3} we obtain that
\begin{align}
(\partial_t |g|^2)(\bar x(t),t) \leq 2B^2 \eps^{-1} \|g(t)\|_{L^\infty}^2.
\label{eq:g:eqn:4}
\end{align}
The pointwise differentiability assumptions further warrant the use of the classical Rademacher theorem (see Appendix) which implies that
\begin{align}
\frac{d}{dt}\|g(t)\|^2_{L^\infty}= (\partial_t |g|^2)(\overline{x}(t),t)
\label{eq:g:eqn:5}
\end{align}
for almost every $t$, where $\bar x = \bar x(t)$ is as above.
From \eqref{eq:g:eqn:4}, \eqref{eq:g:eqn:5}, and the Gr\"onwall inequality it follows that
\[
\|g(t)\|_{L^\infty} \leq \|g(0)\|_{L^\infty} \exp\left( 2 B^2 \eps^{-1}  t \right)
\]
which concludes the proof of the theorem since $g(0) = f_1(0) - f_2(0) =0$.

\appendix

\section{Rademacher Theorem}
Let us recall the classical Rademacher theorem for the convenience of the reader. Suppose $f(x,t) \colon \RR \times \RR_+ \to \RR$ is a Lipschitz in time function uniformly in $x$, i.e. $\exists L>0$ such that for all $t,s,x$ we have
\[
|f(x,t) - f(x,s) | \leq L|t-s|.
\]
Suppose that at any time $t$ there is a point $x(t)$ such that $f(x(t),t)  = M(t) = \sup_x f(x,t)$. Then $M(t)$ is a Lipschitz function with the same constant $L$. Indeed, let $t,s\in \RR_+$ and wlog $M(t)>M(s)$. Then
\[
M(t) - M(s) = f(x(t),t) - f(x(t),s)+\underbrace{f(x(t),s) - f(x(s),s)}_{\leq 0} \leq f(x(t),t) - f(x(t),s) \leq L|t-s|.
\]
Then by the classical Rademacher Theorem in 1D, $M$ is absolutely continuous on any finite interval, i.e.
\[
M(t) - M(s) = \int_s^t m(\tau) d\tau,
\]
where $\|m\|_\infty \leq L$, and hence $M'=m$ a.e. To show that $M'(t) = \partial_t f(x(t),t)$ at the points where $M'$ exists we need extra assumption: for every $x$, $f(x,\cdot)$ is differentiable everywhere in $t$.  Then
\[
\begin{split}
M'(t) &= \lim_{h\to 0+} \frac{f(x(t+h),t+h) - f(x(t), t+h)+f(x(t), t+h)- f(x(t),t)}{h} \\
&\geq \lim_{h\to 0+}\frac{f(x(t), t+h)- f(x(t),t)}{h} = \partial_t f(x(t),t).
\end{split}
\]
Taking $h<0$ proves the opposite inequality.

\bigskip
\begin{center}
{\bf Acknowledgements}
\end{center}
The work of P.C. is supported in part by the NSF grants DMS-1209394 and DMS-1265132. The work of F.G. is partially supported by the  MTM2014-59488-P grant (Spain), the Ram\'on y Cajal program RyC-2010-07094 and by the P12-FQM-2466 grant from Junta de Andaluc\'ia, Spain.
The work of R.S. is partially supported by NSF grants DMS-1210896 and DMS-1515705. He thanks the Princeton University and University of Seville for warm hospitality during preparation of this work. The work of V.V. is supported in part by the NSF grants DMS-1348193 and DMS-1514771, and an Alfred P. Sloan Research Fellowship. He thanks the University of Seville for warm hospitality during preparation of this work.


\begin{thebibliography}{CCG{\etalchar{+}}13a}

\bibitem[A04]{Ambrose04} D.M. Ambrose.\newblock Well-posedness of two-phase Hele-Shaw flow without surface tension. European J. Appl. Math. 15, no. 5, 597-607, 2004.

\bibitem[A07]{Ambrose07}
D.M. Ambrose.
\newblock Well-posedness of two-phase Darcy flow in 3D.
\newblock {\em Quart. Appl. Math.}, 65(1):189--203, 2007.


\bibitem[A14]{Ambrose13}
D.M. Ambrose.
\newblock The zero surface tension limit of two-dimensional interfacial {D}arcy
  flow.
\newblock {\em J. Math. Fluid Mech.}, 16(1):105--143, 2014.

\bibitem[BCGB14]{BerselliCordobaGranero14}
L.C. Berselli, D.~C\'ordoba, and R.~Granero-Belinch\'on.
\newblock Local solvability and turning for the inhomogeneous {M}uskat problem.
\newblock {\em Interfaces Free Bound.}, 16(2):175--213, 2014.

\bibitem[Bea72]{Bear72}
J.~Bear.
\newblock {\em Dynamics of fluids in porous media}.
\newblock Courier Dover Publications, 1972.

\bibitem[BSW14]{BeckSosoeWong14}
T.~Beck, P.~Sosoe, and P.~Wong.
\newblock Duchon-{R}obert solutions for the {R}ayleigh-{T}aylor and {M}uskat
  problems.
\newblock {\em J. Differential Equations}, 256(1):206--222, 2014.

\bibitem[CCF{\etalchar{+}}12]{CastroCordobaFeffermanGancedoLopez12}
A.~Castro, D.~C\'ordoba, C.~Fefferman, F.~Gancedo, and M.~L\'opez-Fern\'andez.
\newblock {R}ayleigh-{T}aylor breakdown for the {M}uskat problem with
  applications to water waves.
\newblock {\em Ann. of Math. (2)}, 175(2):909--948, 2012.

\bibitem[CCFG13a]{CastroCordobaFeffermanGancedo13}
A.~Castro, D.~Cordoba, C.~Fefferman, and F.~Gancedo.
\newblock Breakdown of smoothness for the {M}uskat problem.
\newblock {\em Arch. Ration. Mech. Anal.}, 208(3):805--909, 2013.

\bibitem[CCFG13b]{CastroCordobaFeffermanGancedo13Sub}
A.~Castro, D.~C\'ordoba, C.~Fefferman, and F.~Gancedo.
\newblock Splash singularities for the one-phase {M}uskat problem in stable
  regimes.
\newblock {\em arXiv:1311.7653}, 2013.

\bibitem[CCG{\etalchar{+}}13a]{ConstantinCordobaGancedoRodriguezStrain13}
P.~Constantin, D.~C{{\'o}}rdoba, F.~Gancedo, L.~Rodriguez-Piazza, and R.M.
  Strain.
\newblock On the {M}uskat problem: global in time results in {2D} and {3D}.
\newblock {\em arXiv:1310.0953}, 2013.

\bibitem[CCGS13]{ConstantinCordobaGancedoStrain13}
P.~Constantin, D.~C{{\'o}}rdoba, F.~Gancedo, and R.M. Strain.
\newblock On the global existence for the {M}uskat problem.
\newblock {\em J. Eur. Math. Soc.}, 15:201--227, 2013.

\bibitem[CP93]{ConstantinPugh93}
P.~Constantin and M.~Pugh.
\newblock Global solutions for small data to the {H}ele-{S}haw problem.
\newblock {\em Nonlinearity}, 6(3):393--415, 1993.

\bibitem[CTV15]{ConstantinTarfuleaVicol15}
P.~Constantin, A.~Tarfulea, and V.~Vicol.
\newblock Long time dynamics of forced critical {SQG}.
\newblock {\em Comm. Math. Phys.}, 335(1):93--141, 2015.

\bibitem[CV12]{ConstantinVicol12}
P.~Constantin and V.~Vicol.
\newblock Nonlinear maximum principles for dissipative linear nonlocal
  operators and applications.
\newblock {\em Geom. Funct. Anal.}, 22(5):1289--1321, 2012.

\bibitem[CCG13b]{CordobaCordobaGancedo13}
A.~C\'ordoba, D.~C\'ordoba, and F.~Gancedo.
\newblock Porous media: the {M}uskat problem in three dimensions.
\newblock {\em Anal. PDE}, 6(2):447--497, 2013.

\bibitem[CG07]{CordobaGancedo07}
D.~C{{\'o}}rdoba and F.~Gancedo.
\newblock Contour dynamics of incompressible 3-{D} fluids in a porous medium
  with different densities.
\newblock {\em Comm. Math. Phys.}, 273(2):445--471, 2007.

\bibitem[CG09]{CordobaGancedo09}
D.~C\'ordoba and F.~Gancedo.
\newblock A maximum principle for the {M}uskat problem for fluids with
  different densities.
\newblock {\em Comm. Math. Phys.}, 286(2):681--696, 2009.

\bibitem[CG10]{CordobaGancedo10}
D.~C\'ordoba and F.~Gancedo.
\newblock Absence of squirt singularities for the multi-phase {M}uskat problem.
\newblock {\em Comm. Math. Phys.}, 299(2):561--575, 2010.

\bibitem[CGSZ]{CordobaGomezSerranoZlatos15}
D. C\'ordoba, J. G\'omez-Serrano and A. Zlatos.
A note in stability shifting for the Muskat problem.
To appear in {\em Philos. Trans. R. Soc. Lond. Ser. A Math. Phys. Eng. Sci.}, 2015.

\bibitem[CGBOI14]{CordobaGraneroBelinchonOrive14}
D.~C\'ordoba, R.~Granero-Belinch\'on, and R.~Orive-Illera.
\newblock The confined {M}uskat problem: differences with the deep water
  regime.
\newblock {\em Commun. Math. Sci.}, 12(3):423--455, 2014.

\bibitem[CPC14]{CordobaPernasCastano14}
D.~C{{\'o}}rdoba and T.~Pernas-Casta{\~n}o.
\newblock Non-splat singularity for the one-phase {M}uskat problem.
\newblock {\em arXiv:1409.2483}, 2014.

\bibitem[C93]{Chen93}
X. Chen.
\newblock The hele-shaw problem and area-preserving curve-shortening motions.
\newblock {\em Arch. Ration. Mech. Anal.}, 123(2):117--151, 1993.

\bibitem[CGBS14]{ChengGraneroBelinchonShkoller14}
C.H. Cheng, R.~Granero-Belinch\'on, and S.~Shkoller.
\newblock Well-posedness of the {M}uskat problem with {$H^{2}$} initial data.
\newblock {\em arXiv:1412.7737}, 2014.

\bibitem[Dar56]{Darcy56}
H.~Darcy.
\newblock Les {F}ontaines {P}ubliques de la {V}ille de {D}ijon.
\newblock {\em Dalmont, Paris}, 1856.

\bibitem[EM11]{EscherMatioc11}
J.~Escher and B.-V. Matioc.
\newblock On the parabolicity of the {M}uskat problem: well-posedness,
  fingering, and stability results.
\newblock {\em Z. Anal. Anwend.}, 30(2):193--218, 2011.

\bibitem[ES97]{EscherSimonett97}
J.~Escher and G.~Simonett.
\newblock Classical solutions for {H}ele-{S}haw models with surface tension.
\newblock {\em Adv. Differential Equations}, 2(4):619--642, 1997.

\bibitem[GS13]{GancedoStrain13}
F.~Gancedo and R.M. Strain.
\newblock Absence of splash singularities for {SQG} sharp fronts and the {M}uskat problem.
\newblock {\em Proc. Natl. Acad. Sci.}, 111, no. 2, 635-639, 2014.

\bibitem[GB14]{Granero14}
R.~Granero-Belinch{{\'o}}n.
\newblock Global existence for the confined {M}uskat problem.
\newblock {\em SIAM J. Math. Anal.}, 46(2):1651--1680, 2014.

\bibitem[GSGB14]{GomezSerranoGraneroBelinchon14}
J. G\'omez-Serrano and R. Granero-Belinch\'on.
On turning waves for the inhomogeneous Muskat problem: a computer-assisted proof.
{\em Nonlinearity}, 27(6):1471--1498, 2014.

\bibitem[GHS07]{GuoHallstromSpirn07}
Y.~Guo, C.~Hallstrom, and D.~Spirn.
\newblock Dynamics near unstable, interfacial fluids.
\newblock {\em Comm. Math. Phys.}, 270(3):635--689, 2007.

\bibitem[KM15]{KnupferMasmoudi15}
H.~Kn\"upfer and N.~Masmoudi.
\newblock Darcys flow with prescribed contact angle: Well-posedness and
  lubrication approximation.
\newblock {\em Arch. Ration. Mech. Anal.}, pages 1--58, 2015.

\bibitem[MS13]{MuscaluSchlag13b}
C.~Muscalu and W.~Schlag.
\newblock {\em Classical and multilinear harmonic analysis. Vol. II.}
\newblock Number 138 in Cambridge Studies in Advanced Mathematics. Cambridge
  University Press, Cambridge, 2013.

\bibitem[Mus34]{Muskat34}
M.~Muskat.
\newblock Two fluid systems in porous media. {T}he encroachment of water into
  an oil sand.
\newblock {\em J. Appl. Phys.}, 5(9):250--264, 1934.

\bibitem[Ott99]{Otto99}
F.~Otto.
\newblock Evolution of microstructure in unstable porous media flow: a
  relaxational approach.
\newblock {\em Comm. Pure Appl. Math.}, 52(7):873--915, 1999.

\bibitem[SCH04]{SiegelCaflischHowinson04}
M.~Siegel, R.E. Caflisch, and S.~Howison.
\newblock Global existence, singular solutions, and ill-posedness for the
  {M}uskat problem.
\newblock {\em Comm. Pure Appl. Math.}, 57(10):1374--1411, 2004.

\bibitem[ST58]{SaffmanTaylor58}
P.G. Saffman and G.~Taylor.
\newblock The penetration of a fluid into a porous medium or {H}ele-{S}haw cell
  containing a more viscous liquid.
\newblock {\em Proc. Roy. Soc. London. Ser. A}, 245:312--329. (2 plates), 1958.

\bibitem[Sz{\'e}12]{Szekelyhidi12}
L.~Sz{\'e}kelyhidi, Jr.
\newblock Relaxation of the incompressible porous media equation.
\newblock {\em Ann. Sci. {\'E}c. Norm. Sup{\'e}r. (4)}, 45(3):491--509, 2012.

\end{thebibliography}

\newcommand{\etalchar}[1]{$^{#1}$}


\end{document}